\documentclass[reqno]{amsart}
\usepackage[utf8]{inputenc}
\usepackage[a4paper]{geometry}
\usepackage{xargs}                      % Use more than one optional parameter in a new commands
\usepackage[pdftex,dvipsnames]{xcolor}  % Coloured text etc.
\usepackage[colorinlistoftodos,prependcaption,textsize=tiny]{todonotes} % to remove todos at the end, simply add "disable" as an argument in the square brackets above
\makeatletter
\providecommand\@dotsep{5}
\renewcommand{\listoftodos}[1][\@todonotes@todolistname]{%
  \@starttoc{tdo}{#1}}
\makeatother

\usepackage{mathtools, amsfonts, amsthm, amssymb, mathrsfs}
\usepackage{dsfont}
\usepackage[shortlabels]{enumitem}
\setlist[enumerate]{nosep}

\numberwithin{equation}{section}
\theoremstyle{plain}
\newtheorem{theorem}{Theorem}[section]

\newtheorem{lemma}[theorem]{Lemma}
\newtheorem{proposition}[theorem]{Proposition}
\newtheorem{assumption}[theorem]{Assumption}
\newtheorem*{theorem*}{Theorem}

\newtheorem{openproblem}{Open problem}
\newtheorem{definition}[theorem]{Definition}
\theoremstyle{definition}
\newtheorem{example}[theorem]{Example}

\newtheorem{remark}[theorem]{Remark}

\newcommand{\R}{\mathbb{R}}
\newcommand{\N}{\mathbb{N}}
\newcommand{\Z}{\mathbb{Z}}

\newcommand{\X}{\mathcal{X}}
\newcommand{\Y}{\mathcal{Y}}

\newcommand{\PP}{\mathcal{P}}

\newcommand{\di}{\mathrm{d}}

\DeclareMathOperator{\Int}{int}
\DeclareMathOperator{\proj}{proj}

\DeclareMathOperator{\conv}{conv}
\DeclareMathOperator{\id}{Id}

\DeclareMathOperator{\supp}{spt}
\DeclareMathOperator{\spt}{spt}
\DeclareMathOperator{\diam}{diam}

\DeclareMathOperator{\loc}{loc}

\DeclareMathOperator*{\argmin}{arg\,\min}

\usepackage{derivative}
\usepackage[shortlabels]{enumitem}
\setlist[enumerate]{nosep}

\newcommand{\qedwhite}{\hfill \ensuremath{\square}}

\usepackage[backend=biber]{biblatex} % citestyle=alphabetic bibstyle=authortitle
\usepackage{tikz}
\usetikzlibrary{shapes.misc}
\usetikzlibrary{arrows.meta,calc}
\usetikzlibrary{decorations.pathreplacing}
\usetikzlibrary{positioning}

\tikzset{cross/.style={cross out, draw=black, minimum size=2*(#1-\pgflinewidth), inner sep=0pt, outer sep=0pt},
cross/.default={1pt}}

\usepackage{caption}
\usepackage{xcolor}

\definecolor{beamerblue}{RGB}{51,51,178}
\definecolor{targetcolour}{RGB}{51,51,178}
\definecolor{sourcecolour}{RGB}{216, 0, 240}
\definecolor{mixed}{RGB}{255,0,0}
\definecolor{amaranth}{rgb}{0.9, 0.17, 0.31}
\definecolor{amber}{rgb}{1.0, 0.75, 0.0}
\definecolor{amber2}{rgb}{1.0, 0.49, 0.0}

\usepackage{nomentbl}

\usepackage{hyperref}

\title{Quantitative Uniqueness of Kantorovich Potentials}
\author{William Ford$^1$}
\thanks{$^1$CMAP École Polytechnique, Palaiseau, France. Email: \texttt{william.ford@polytechnique.edu}}
\date{}

\newcounter{step}

\begin{document}

\maketitle
\vspace{-0.7cm}
\begin{abstract}
    This paper studies the uniqueness of solutions to the dual optimal transport problem, both qualitatively and quantitatively (bounds on the diameter of the set of optimisers). 
    
    On the qualitative side, we prove that when one marginal measure's support is rectifiably connected (path-connected by rectifiable paths), the optimal dual potentials are unique up to a constant. This represents the first uniqueness result applicable even when both marginal measures are concentrated on lower-dimensional subsets of the ambient space, and also applies in cases where optimal potentials are nowhere differentiable on the supports of the marginals.
    
    On the quantitative side, we control the diameter of the set of optimal dual potentials by the Hausdorff distance between the support of one of the marginal measures and a regular connected set. In this way, we quantify the extent to which optimisers are almost unique when the support of one marginal measure is almost connected. This is a consequence of a novel characterisation of the set of optimal dual potentials as the intersection of an explicit family of half-spaces.
    
    \vspace{0.2cm}
    
\end{abstract}

\setcounter{tocdepth}{1}
\tableofcontents
\vspace{-0.9cm}
\section{Introduction}
Given two probability measures $\rho$ and $\mu$ on $\R^d$ and a cost function $c : \R^d \times \R^d \to \R$, the optimal transport problem between $\rho$ and $\mu$ is the minimisation problem
\begin{equation}
\label{intro: c OT problem}
    \min_{\gamma \in \Pi(\rho, \mu)} \int_{\R^d \times \R^d} c(x,y) \di \gamma(x, y),
\end{equation}
where $\Pi(\rho, \mu)$ denotes the set of all couplings of $\rho$ and $\mu$, i.e. probabilities over $\R^d \times \R^d$ with first marginal $\rho$ and second marginal $\mu$. Problem \eqref{intro: c OT problem} admits the dual formulation
\begin{equation}
\label{intro: c dual problem}
    \max_{\substack{\phi, \psi \in L^1(\rho) \times L^1(\mu) \\ \phi \oplus \psi \leq c}} \int_{\R^d} \phi(x) \di \rho(x) + \int_{\R^d} \psi(y) \di \mu(y),
\end{equation}
where $\phi \oplus \psi \leq c$ denotes the pointwise inequality constraint
\begin{equation}
    \label{intro: dual constraint}
    \phi(x) + \psi(y) \leq c(x, y) \quad \forall (x,y) \in \R^d \times \R^d.
\end{equation}
Optimisers of the dual problem \eqref{intro: c dual problem} are referred to as \textit{Kantorovich potentials}. The principal aim of this paper is to address the following questions:
 \begin{enumerate}[label= (\Roman*)]
    \item Under which hypotheses on $\rho$, $\mu$, and $c$ are Kantorovich potentials unique? \label{q: qual unique}
    \item Under which hypotheses on $\rho$, $\mu$, and $c$ are Kantorovich potentials close to being unique, i.e. the diameter of the set of optimal potentials is quantitatively small? \label{q: quant unique}
\end{enumerate}
These questions are of fundamental interest for understanding the well-posedness of the dual optimal transport problem \eqref{intro: c dual problem}. We also expect these results to be a natural, necessary tool for a more thorough investigation of the quantitative stability properties of both problems \eqref{intro: c OT problem} and \eqref{intro: c dual problem}. The majority of current quantitative stability results make assumptions ensuring both primal and dual uniqueness; see, for example, \cite{delalande2023quantitative, letrouit2024gluing, mischler2024quantitative, letrouit2025lectures}.

Question \ref{q: qual unique} has been the subject of numerous recent investigations, of which we provide an overview in Section \ref{Sect: Related works}. Kantorovich potentials are at best unique up to a constant, since for any admissible $(\phi, \psi)$, the potentials $\{(\phi-l, \psi+ l)\}_{l \in \R}$ are also admissible, and attain the same value in the objective.

To the best of our knowledge, Question \ref{q: quant unique} has not been investigated, and our results provide a first response. We denote the set of $\rho$-side Kantorovich potentials by
\begin{equation*}
    \Phi_c(\rho, \mu) := \left\{ \phi \in  L^1(\rho) : \exists \psi \in L^1(\mu) \text{ with } (\phi, \psi) \text{ optimal for } \eqref{intro: c dual problem}\right\}.
\end{equation*}
Let $p \in [1, +\infty]$. Our results measure quantitative uniqueness by the $L^p(\rho)$ diameter of the optimal $\rho$-side set of potentials, modulo additive constants:
\begin{equation}
\label{eq: diameter definition}
    \diam_{L^p}(\Phi_c(\rho, \mu)) := \sup_{\phi_0, \phi_1 \in \Phi_c(\rho, \mu)} \inf_{l \in \R} \| \phi_0 - \phi_1 - l\|_{L^p(\rho)}.
\end{equation}
Small diameter means that the difference between any two $\rho$-side Kantorovich potentials is a function of small oscillation. We do not explicitly study the uniqueness of pairs $(\phi, \psi)$ of optimisers for \eqref{intro: c dual problem}, but any bounds on $\diam_{L^\infty} (\Phi_c(\rho, \mu))$ implicitly imply the same bound for the set of $\mu$-side Kantorovich potentials, see Section \ref{subsect: equivalence of pairwise and single potential unique}.

\subsection{Contributions to qualitative uniqueness}
Our main qualitative uniqueness result establishes that Kantorovich potentials are unique up to a constant when one of the measures has rectifiably connected support. 
\begin{definition}
A set $\Omega \subseteq \R^d$ is said to be rectifiably connected if for any $x, x' \in \Omega$, there exists $\omega \in W^{1,1}([0,1]; \Omega)$ with $\omega(0) = x'$ and $\omega(1) = x$. After a reparametrisation by arc length, one can always choose such a curve to be Lipschitz.
\end{definition}
\begin{theorem}
\label{thrm: uniqueness rect connected bounded}
Let $\rho$ and $\mu$ be probability measures on $\R^d$ such that $\supp \rho$ is rectifiably connected and $\supp \mu$ is bounded. Suppose that $c \in \mathcal{C}^1(\R^d \times \R^d)$ and that there exist real-valued upper semi-continuous functions $a \in L^1(\rho)$ and $b \in L^1(\mu)$ such that
    \begin{equation*}
       \forall (x, y) \in \R^d, \quad |c(x, y)| \leq a(x) + b(y).
    \end{equation*}
Then optimisers of the dual problem \eqref{intro: c dual problem} exist and are unique up to a constant $\rho$ and $\mu$ a.e.
\end{theorem}
Measures with rectifiably connected support include, for example, measures whose support is a $\mathcal{C}^1$ connected submanifold of $\R^d$. To the best of our knowledge, Theorem \ref{thrm: uniqueness rect connected bounded} represents the first dual uniqueness result for which both marginal measures can be concentrated on lower-dimensional subsets of the ambient space. Theorem \ref{thrm: uniqueness rect connected bounded} applies to the cost $c(x, y) = \|x - y \|^p$ for all $p \in (1, +\infty)$, but not $p=1$.

By writing each marginal measure's support as a union of rectifiably connected components, we obtain a graph-based sufficient condition for uniqueness.
\begin{theorem}
\label{thrm: uniqueness graph continuous case}
 Let $c \in \mathcal{C}^1(\R^d \times \R^d)$.  Let $\rho$ and $\mu$ be probability measures on $\R^d$ with bounded support. Assume that $(A_i)_{i \in I}$ and $(B_j)_{j \in J}$ are families of measurable, rectifiably connected sets such that
 \begin{equation*}
     \spt \rho = \bigcup_{i \in I} A_i \quad \text{ and } \quad  \spt \mu = \bigcup_{j \in J} B_j.
 \end{equation*}
 Then solutions to the dual problem \eqref{intro: c dual problem} exist, and are unique up to a constant $\rho$ and $\mu$ a.e. if the bipartite graph $G = (V, E)$ is connected, where
 \begin{enumerate}[label=(\roman*)]
     \item $V = (A_i)_{i \in I} \cup (B_j)_{j \in J}$.
     \item $(A_i, B_j) \in E$ if and only if there exists $\gamma \in \Pi(\rho, \mu)$ optimal for \eqref{intro: c OT problem} with $\gamma(A_i \times B_j) > 0.$
 \end{enumerate}
\end{theorem}

\subsection{Contributions to quantitative uniqueness}

For two sets $A, B \subseteq \R^d$, the Hausdorff distance $d_H(A, B)$ is given by
\begin{equation*}
    d_H(A, B) := \max\left\{\sup_{x \in A} d(x, B), \sup_{x \in B} d(x, A)\right\}.
\end{equation*}
Our main quantitative uniqueness result bounds $\diam_{L^\infty}\Phi_c(\rho, \mu)$ by the Hausdorff distance between $\spt \rho$ and a connected set $ \Omega \subseteq \R^d$ satisfying the following assumption.
\begin{assumption}
\label{ass: domain assumption connexivity by bounded curvature arcs}
    Suppose $\Omega \subseteq \R^d$ is such that
    \begin{equation}
\label{eq: connexivity of support by arcs of uniformly bounded curvature}
    C_\Omega := \sup_{z, z' \in \Omega} \inf_{\substack{\omega : [0, 1] \to \Omega \\ \omega(0) = z',\, \omega(1) = z \\ \dot \omega \in BV([0,1]; \R^d)}} \| \dot \omega\|_{L^\infty([0, 1])} + |\ddot \omega|([0, 1]) < + \infty.
\end{equation}
\end{assumption}
\noindent Here $BV([0,1]; \R^d)$ denotes the space of functions whose derivative is a finite, vector-valued Radon measure, and $|\ddot \omega|$ denotes the variation measure of $\ddot \omega$, see \cite[Chapter 5]{evans2025measure}. Roughly speaking, such sets are rectifiably connected, using arcs of uniformly bounded length and curvature. Examples include bounded $\mathcal{C}^0$ domains (domains whose boundary is locally the graph of a continuous function) and $\mathcal{C}^{1,1}$ compact connected submanifolds with boundary, see Appendix \ref{ch: path assumption}.
\begin{theorem}
\label{thrm: c 1 alpha cost Linfty bound}
Let $\X , \Y \subseteq \R^d$, let $c \in \mathcal{C}^{1, \alpha}(\X \times \Y)$ for some $\alpha \in (0, 1]$. Then there exists a constant $C(\X, \Y, c) >0$ such that for all $\Omega \subseteq \X$ satisfying Assumption \ref{ass: domain assumption connexivity by bounded curvature arcs}, all $\rho \in \mathcal{P}(\X)$, and all $\mu \in \mathcal{P}(\Y)$, $\Phi_c(\rho,\mu)$ is non-empty and
    \begin{equation}
        \diam_{L^\infty}(\Phi_c(\rho, \mu)) \leq C(1+C_\Omega)^2 \left(d_H(\supp \rho, \Omega)^{1- 1/(\alpha + 1)} + d_H(\supp \rho, \Omega)\right),
    \end{equation}
    where $C_\Omega\geq0$ is the quantity defined in Assumption \ref{ass: domain assumption connexivity by bounded curvature arcs}.
\end{theorem}

\noindent Here, the Hölder space $\mathcal{C}^{1, \alpha}(\X)$ consists of all functions $f : \X \to \R$ continuously differentiable such that
\begin{equation*}
    \|f\|_{\mathcal{C}^0(\X)} + \| \nabla f\|_{\mathcal{C}^0(\X)} + [\nabla f]_{\X, \alpha} < +\infty,
\end{equation*}
where
\begin{equation*}
    \|f\|_{\mathcal{C}^0(\X)} := \sup_{x \in \X} \|f(x)\| \quad \text{ and } \quad [\nabla f]_{\X, \alpha} := \sup_{\substack{x_0, x_1 \in \X \\ x_0 \neq x_1}} \frac{\|\nabla f(x_0) - \nabla f(x_1)\|}{\|x_0 - x_1\|^\alpha}.
\end{equation*}
In particular, when $\X, \Y \subseteq \R^d$ are both compact, Theorem \ref{thrm: c 1 alpha cost Linfty bound} applies to the cost $c(x, y) = \|x - y\|^p$ for all $p \in [2, +\infty)$ with $\alpha = 1$, giving an exponent $1/2$ for $d_H(\spt \rho, \Omega)$ at small scales. When $p \in (1, 2]$, Theorem \ref{thrm: c 1 alpha cost Linfty bound} applies with $\alpha = p-1$ and exponent $1-1/p$ for $d_H(\spt \rho, \Omega)$, since the function $z \mapsto \|z\|^p$ is $\mathcal{C}^{1, p-1}_{\loc}$. (We provide a short proof due to \cite{mischler2024quantitative} of this fact in Appendix \ref{app: regularity of p cost}.) We do not know if these exponents are optimal. In Section \ref{sect: optimality of exponent in grid case}, we provide an example showing that, in general, we cannot expect a better exponent than $1$.

Bounded convex sets $\Omega \subseteq \X$ satisfy Assumption \ref{ass: domain assumption connexivity by bounded curvature arcs} with constant $C_\Omega = \diam \Omega \leq \diam \X$, where $\diam$ denotes the diameter of a set in $\R^d$. Thus, for $\X$ bounded, taking $\Omega = \conv \spt \rho$ the convex hull, we obtain the estimate
\begin{equation}
   \forall \rho \in \PP(\X),\, \forall \mu \in \PP(\Y), \quad \diam_{L^\infty}(\Phi_c(\rho, \mu)) \leq C(\X, \Y, c)\, d_H(\supp \rho, \conv \spt \rho)^{1-1/(\alpha + 1)}.
\end{equation}
All sets $\Omega \subseteq \R^d$ satisfying Assumption \ref{ass: domain assumption connexivity by bounded curvature arcs} are bounded. To show that $\Omega$ need not be bounded for a quantitative uniqueness result to hold, we present Theorem \ref{thrm: quant L^p moment bounds} below. We control the $L^q$ diameter of $\Phi_c(\rho, \mu)$ for $q \in [1, + \infty)$, for $\rho$ satisfying uniform $q/2$ moment bounds. For simplicity of exposition, we state and prove this only for quadratic cost $c(x, y) = \|x-y\|^2$. An analogue of this result would also hold for $c \in \mathcal{C}^{1, \alpha}_{\loc}(\R^d \times \R^d)$ satisfying growth conditions of the form
\begin{equation*}
    \| \nabla_x c\|_{\mathcal{C}^0(B_R \times \Y)},\, [\nabla_{x}c]_{B_{R} \times \Y, \alpha} \leq C(1 + R^s)
\end{equation*}
for some $C, s>0$ and $B_R = B_R(0) \subseteq \R^d$, but we do not prove this here. A point $z' \in \R^d$ is a centre of the star-shaped set $\Omega$ if for any $z \in \Omega$, we have $[z', z] \subseteq \Omega$.
\begin{theorem}
\label{thrm: quant L^p moment bounds}
    Let $c(x, y) =\|x-y\|^2$. Let $\Omega \subseteq \R^d$ be a star-shaped set with centre at $z' \in \R^d$, and let $\Y \subseteq \R^d$ be compact. Let $q \geq 1$. Then for each $M>0$ there exists a constant $C(M, z',  \Y) > 0$ such that for all $\rho \in \mathcal{P}_{2}(\R^d)$ with moment $\int \|x\|^{q/2} \di\rho \leq M$, and all $\mu \in \PP(\Y)$, $\Phi_c(\rho,\mu)$ is non-empty and
    \begin{equation}
        \diam_{L^{q}}(\Phi_c(\rho, \mu)) \leq C \left(d_H(\spt \rho, \Omega)^{1/2} + d_H(\spt \rho, \Omega)\right).
    \end{equation}
\end{theorem}
It is not clear to us that the moment bounds appearing in Theorem \ref{thrm: quant L^p moment bounds} are necessary for quantitative uniqueness with unbounded source measures. These appear as a by-product of the proof, which involves integrating the bound derived for the proof of Theorem \ref{thrm: c 1 alpha cost Linfty bound}.

In applications, one is interested in regular measures $\rho$ for which uniqueness holds. If one does not know $\rho$ explicitly, one either works with a grid approximation supported on $\spt \rho \cap \varepsilon \mathbb{Z}^d$, or if one can obtain independent and identically distributed (i.i.d.) samples $\{X_i\}_{i=1}^N \sim \rho$, it is often convenient to work with the empirical measure
\begin{equation*}
    \rho_N = \frac{1}{N} \sum_{i=1}^N \delta_{X_i}
\end{equation*}
as an approximation for $\rho$. Theorem \ref{thrm: c 1 alpha cost Linfty bound} already applies directly to the grid case, but we also present Theorem \ref{thrm: quantitative uniqueness for measures on uniform grids} below, which shows that in some cases, we can attain a better exponent. In particular, the proof is specific to quadratic cost $c(x, y) = \|x-y\|^2$, whilst applying to potentially unbounded source measures, without any uniform moment bounds. It is optimal in its exponent, as a consequence of Section \ref{sect: optimality of exponent in grid case}.
\begin{theorem} 
\label{thrm: quantitative uniqueness for measures on uniform grids}
Let $c(x, y) = \|x-y\|^2$. Let $ I_1,..., I_d \subseteq \R $ be (potentially unbounded) intervals, let
\begin{equation*}
    \Omega = I_1 \times \cdots \times I_d \subseteq \R^d,
\end{equation*}
and let $\Y \subseteq \R^d$ be compact. Then for all $ \rho \in \PP_2(\R^d)$ satisfying $\supp \rho = \Omega \cap \varepsilon\Z^d$ for some $\varepsilon > 0$, and all $\mu \in \mathcal{P}(\Y)$, $\Phi_c(\rho,\mu)$ is non-empty and
\begin{equation*}
        \diam_{L^\infty}(\Phi_c(\rho, \mu)) \leq C \varepsilon.
\end{equation*}
where $C = 2d \diam(\Y)$.
\end{theorem}
The conclusion of Theorem \ref{thrm: quantitative uniqueness for measures on uniform grids} holds (with a different constant) for a general class of potentially unbounded sets $\Omega \subseteq \R^d$ satisfying the combinatorial path condition Assumption \ref{ass: domain assumption for grid discretisation}, see Section \ref{ch: quant uniqueness}. Many other unbounded sets in $\R^d$ satisfy Assumption \ref{ass: domain assumption for grid discretisation}; our example was chosen for its simplicity and relevance to common applications.

In the empirical measure case, we control the expected diameter and concentration properties of $\Phi_c(\rho_N, \mu)$ for large $N$. This is a direct consequence of expectation and concentration bounds from \cite{reznikov2016covering} for the quantity
\begin{equation*}
    d_H(\spt \rho, \{X_i\}_{i=1}^n)= d_H(\spt \rho, \spt \rho_N),
\end{equation*}
applied directly to the conclusion of Theorem \ref{thrm: c 1 alpha cost Linfty bound}. We again consider general costs $c \in \mathcal{C}^{1, \alpha}$.
\begin{theorem}
\label{thrm: quant stochastic}
Let $\X, \Y \subseteq \R^d$, and suppose $c \in \mathcal{C}^{1, \alpha}(\X \times \Y)$ for some $\alpha \in (0, 1]$. Let $\rho$ be a probability measure on $\X$ such that
\begin{enumerate}[label=(\roman*)]
    \item $\spt \rho$ satisfies Assumption \ref{ass: domain assumption connexivity by bounded curvature arcs}.
    \item There exist constants $C, s, r_0>0$ such that for all $x \in \spt \rho$ and $r \leq r_0$,
    \begin{equation}
    \label{eq: ahlfors regularity}
        \rho(B_r(x)) \geq Cr^s.
    \end{equation}
\end{enumerate}
Let $\{X_i\}_{i=1}^N \sim \rho$ be i.i.d. random variables and let $\rho_N = \frac{1}{N} \sum_{i=1}^N \delta_{X_i}$ be the corresponding empirical measure. Then there exist constants $C_0, C_1, C_2, C_3, \beta_0 >0$ depending on $c, \X, \Y$, and $\rho$ such that for all $\mu \in \PP(\Y)$, $\Phi_c(\rho_N,\mu)$ is non-empty and
\begin{equation}
\label{eq: expectation bound diameter general case}
    \forall N \in \mathbb{N}, \quad\mathbb{E} \left[\diam_{L^\infty}(\Phi_c(\rho_N, \mu)) \right] \leq C_0  \left(\frac{ \log{N}}{N}\right)^{\frac{\alpha}{s(1+\alpha)}},
\end{equation}
and for all $\beta > \beta_0$
\begin{equation}
\label{eq: concentration bound for the diameter}
    \forall N \in \mathbb{N}, \quad\mathbb{P}\left[ \diam_{L^\infty}(\Phi_c(\rho_N, \mu)) \geq C_1 \left(\frac{\beta \log{N}}{N}\right)^{\frac{\alpha}{s(1+\alpha)}}\;\right] \leq C_2 N^{1-C_3 \beta}.
\end{equation}
\end{theorem}
Examples of measures to which Theorem \ref{thrm: quant stochastic} applies include:
\begin{itemize}
    \item Absolutely continuous $\rho$, with density uniformly bounded below with respect to the uniform measure on a bounded Lipschitz domain. Indeed, bounded Lipschitz domains satisfy Assumption \ref{ass: domain assumption connexivity by bounded curvature arcs} by Proposition \ref{prop: path regularity for various domains}, and $\rho$ satisfies \eqref{eq: ahlfors regularity} with $s=d$, since the restriction of the Lebesgue measure to bounded Lipschitz domains satisfies \eqref{eq: ahlfors regularity}.
    \item When $\spt \rho$ is a $\mathcal{C}^{1,1}$ compact connected submanifold of $\R^d$ with boundary, and $\rho$ is absolutely continuous with respect to the uniform surface measure on $\spt \rho$, with density uniformly bounded below. Indeed, Assumption \ref{ass: domain assumption connexivity by bounded curvature arcs} is satisfied for $\mathcal{C}^{1,1}$ compact connected submanifolds, and $\rho$ satisfies \eqref{eq: ahlfors regularity} with $s = \dim \spt \rho$, since the uniform measure of a bounded $s$-dimensional $\mathcal{C}^{1, 1}$ submanifold satisfies \eqref{eq: ahlfors regularity}.
\end{itemize}

\subsection{Characterisation of the set of Kantorovich potentials}
The main tool used to prove our quantitative uniqueness results is a novel characterisation of the optimal set of potentials $\Phi_c(\rho, \mu)$, inspired by the graph-based qualitative uniqueness result, Theorem \ref{thrm: uniqueness graph continuous case}. We sketch the idea below for the Brenier potentials associated with the bilinear cost $c_l(x, y) = -\langle x, y \rangle$, see Section \ref{ch: Kantorovich potentials} for precise definitions. The compatibility condition \eqref{intro: compatibility condition} between primal and dual optimisers says that
\begin{equation*}
    \Phi_{c_l}(\rho, \mu) =\{ \phi \text{ convex} : \spt \gamma \subseteq \partial \phi \text{ for all } \gamma \text{ optimal for } \eqref{intro: c OT problem}\}.
\end{equation*}
where $\partial\phi$ denotes the graph of the subgradient of the convex function $\phi$. Fix some $\gamma \in \Pi(\rho, \mu)$ an optimiser for \eqref{intro: c OT problem}, and take points $\{(x_i, y_i)\}_{i=0}^n \subseteq \spt \gamma$. For any $\phi \in \Phi_{c_l}(\rho, \mu)$, subgradient inequalities give that for any $i = 0,...,n-1$,
\begin{equation}
\label{eq: upper and lower bound subsdiff inequalities}
\langle x_{i+1} - x_i, y_i \rangle \leq  \phi(x_{i+1})- \phi(x_i) \leq \langle x_{i+1} - x_i, y_{i+1} \rangle.
\end{equation}
The lower bound corresponds to if the gradient of $\phi$ on the interval $[x_{i}, x_{i+1}] \subseteq \R^d$ was the component of $y_i$ in the direction $x_{i+1} -x_i$, whilst the upper bound corresponds to that of $y_{i+1}$, see Figure \ref{fig:upper and lower bounds via subgradient}.
\begin{figure}[ht]
    \centering
\begin{tikzpicture}[scale=0.9, >=Latex]

\def\xzero{0.0}
\def\xone{3.2}
\def\yzero{0.25} 
\def\yone{1}   

\tikzset{
  ax/.style={->, line width=0.7pt},
  bound/.style={dashed, line width=0.9pt},
  mid/.style={line width=1.1pt},
  dot/.style={circle, fill, inner sep=1.2pt},
}

\pgfmathdeclarefunction{philow}{1}{\pgfmathparse{\yzero*(#1-\xzero)}}
\pgfmathdeclarefunction{phiup}{1}{\pgfmathparse{\yone*(#1-\xzero)}}
\pgfmathdeclarefunction{phimid}{1}{\pgfmathparse{0.125*(#1 +1)*(#1 + 1)- 0.125}}

\pgfmathsetmacro{\philowxone}{philow(\xone)}
\pgfmathsetmacro{\phiupxone}{phiup(\xone)}
\pgfmathsetmacro{\phimidxone}{phimid(\xone)}

  \draw[ax] (-0.4,0) -- (4.2,0) node[below] {$x$};
  \draw[ax] (0,-0.4) -- (0,4.0) node[left] {$\phi(x)$};

  \draw (\xzero,0) ++(0,-0.08) -- ++(0,0.16) node[below=3pt] {$x_i$};
  \draw (\xone,0)  ++(0,-0.08) -- ++(0,0.16) node[below=3pt] {$x_{i+1}$};

  \draw[bound] (\xzero,0) -- (\xone,\philowxone);
  \node[below left] at (2,{philow(2)}) {$y_{i}$};

  \draw[bound] (\xzero,0) -- (\xone,\phiupxone);
  \node[above left] at (1.5,{phiup(1.25)}) {$y_{i+1}$};

  \draw[mid] plot[smooth, domain=\xzero:\xone] (\x,{phimid(\x)});

  \node[dot, label=above left:{$(x_i,\phi(x_i))$}] at (\xzero,0) {};
  \node[dot, label=right:{$(x_{i+1},\phi(x_{i+1}))$}] at (\xone,\phimidxone) {};

  \draw[line width=0.8pt] (\xone,\philowxone) -- (\xone,\phiupxone);
  \draw (\xone,\philowxone) ++(-0.10,0) -- ++(0.20,0);
  \draw (\xone,\phiupxone)  ++(-0.10,0) -- ++(0.20,0);
\end{tikzpicture}
    \caption{Interval of admissibility for $\phi(x_{i+1})- \phi(x_i)$.}
    \label{fig:upper and lower bounds via subgradient}
\end{figure}
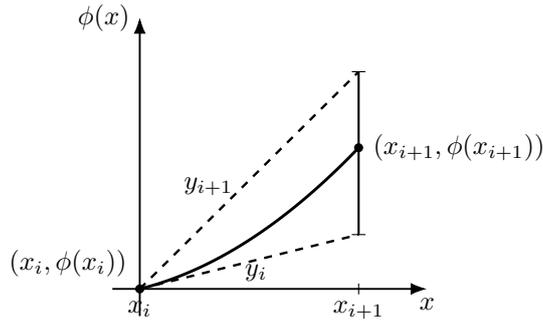
Summing the inequalities \eqref{eq: upper and lower bound subsdiff inequalities}, we deduce
\begin{equation}
\label{eq: demonstration of quant argument for quadratic in intro}
    \sum_{i=0}^{n-1} \langle x_{i+1} - x_i, y_{i} \rangle\leq \phi(x_n) - \phi(x_0) \leq \sum_{i=0}^{n-1} \langle x_{i+1} - x_i, y_{i+1} \rangle.
\end{equation}
We define the quantity $\lambda(x, x')$ as the infimum of the right-hand side of \eqref{eq: demonstration of quant argument for quadratic in intro}, over all possible such pairs of points $\{(x_i, y_i)\}_{i=0}^n \subseteq \spt \gamma$ for any optimal plan $\gamma$, such that $x_0 = x'$ and $x_n = x$. Then for any $\phi \in \Phi_{c_l}(\rho, \mu)$,
\begin{equation}
\label{eq: intro admissible interval}
    \forall x, x' \in \spt \rho, \quad  \phi(x) - \phi(x') \in [-\lambda(x', x), \lambda(x, x')].
\end{equation}
Each constraint $\phi(x) - \phi(x') \leq \lambda(x, x')$ is a half space, supported by a hyperplane orthogonal to the signed linear form $\delta_{x'} - \delta_x$, a difference of Dirac masses. For each $x, x' \in \spt \rho$ we have two parallel such supporting hyperplanes, giving the admissible interval \eqref{eq: intro admissible interval}. The optimal set $\Phi_{c_l}(\rho, \mu)$ is the intersection of these half spaces. 

For any pair $\phi_0, \phi_1 \in \Phi_{c_l}(\rho, \mu)$, fixing some $x' \in \spt \rho$, we can assume $\phi_0(x') = \phi_1(x')$ up to choosing $l \in \R$ in the diameter definition \eqref{eq: diameter definition}. It follows that
\begin{equation*}
    \forall x \in \spt \rho, \quad |\phi_0(x) - \phi_1(x)| = \big| (\phi_0(x) - \phi_0(x')) - (\phi_1(x) - \phi_1(x'))\big| \leq \lambda(x, x') + \lambda(x', x).
\end{equation*}
Our quantitative uniqueness results are thus proven by showing that for any pair $x, x' \in \spt \rho$, the length of the interval \eqref{eq: intro admissible interval} is uniformly small. We do this by choosing good candidate sets of points $\{(x_i, y_i)\}_{i=0}^n$ to test in the infimum definition of $\lambda$.

\subsection{Related works}
\label{Sect: Related works}

The main result underpinning previous uniqueness theorems for the dual problem is \cite[Proposition 7.18]{santambrogio2015optimal}, whose proof also appears earlier as a remark in \cite[Remark 10.30]{villani2008optimal}. The fundamental assumptions are that $\spt \rho$ is the ``closure of a connected open set", and $\spt \mu$ is bounded. Sets which are the ``closure of a connected open set" are inherently locally full-dimensional, unlike the hypothesis of rectifiable connectivity appearing in Theorem \ref{thrm: uniqueness rect connected bounded}. In \cite[Corollary 4 + Lemma 8]{staudt2025uniqueness}, the authors show that for some costs, if $\spt \rho$ satisfies the ``closure of a connected open set" hypothesis, we can remove the hypothesis that $\spt \mu$ is bounded. Their proof essentially involves deducing that the only points in $\spt \rho$ which ``send mass towards infinity" under some optimal plan $\gamma \in \Pi(\rho, \mu)$ are those on the boundary $\partial \spt \rho$, see Remark \ref{rmk: spt mu bounded}.

For the fully discrete dual transport problem, recent works have given combinatorial conditions for uniqueness. In this case, uniqueness is understood in the language of linear programming and extreme points of convex polyhedra. In \cite[Proposition 3.5 (ii)/Corollary 3.14]{acciaio2025characterization}, the authors provide a discrete version of our graph-based result, Theorem \ref{thrm: uniqueness graph continuous case}, whose proof is intimately related to the older work \cite{balinski1984faces}, which studies the extreme points of the dual transport polyhedra.

Recent works \cite{staudt2025uniqueness, yang2023optimal} have partially adapted the discrete combinatorial results to the continuous setting, obtaining weaker forms of our Theorem \ref{thrm: uniqueness graph continuous case}. They use an asymmetric ``closure of a connected open set" style support hypothesis for each connected component of the supports of marginal measures rather than their full supports.

In this paper, we only consider the classical optimal transport problem, with no regularisation. For entropic optimal transport, the dual objective is strictly concave, so the uniqueness of entropic (Schrödinger) potentials always holds up to a constant without needing support connectedness hypotheses. Dual uniqueness for the quadratically regularised optimal transport problem has also been studied, in \cite{nutz2025quadratically}. Here, similar conditions and hypotheses are given as in the unregularised case.

\subsection{Open problems}
We leave the following natural questions open concerning the qualitative uniqueness of Kantorovich potentials.
\begin{openproblem}
\label{problem: unbounded target}
    For which costs can we remove the hypothesis in Theorem \ref{thrm: uniqueness rect connected bounded} that $\spt \mu$ is bounded?
\end{openproblem}
In the proof of Theorem \ref{thrm: uniqueness rect connected bounded}, the hypothesis that $\spt \mu$ is bounded is principally used to ensure that $\rho$-side Kantorovich potentials are locally Lipschitz. We do not know if this hypothesis can be removed in general, see Remark \ref{rmk: spt mu bounded}.
\begin{openproblem}
    Does Theorem \ref{thrm: uniqueness rect connected bounded} hold if we only suppose $\spt \rho$ is connected, not rectifiably connected?
\end{openproblem}
It is not clear if ``rectifiably connected" is the correct limiting hypothesis for uniqueness. For example, let $f: [0, 1] \to [0, 1]$ be a sample path of Brownian motion, and consider $\rho = (\id, f)_\#\text{Leb}|_{[0,1]}$, where $\text{Leb}$ is the Lebesgue measure. Then $\supp \rho$ is path connected, but every path connecting two points is locally of infinite length and nowhere differentiable, so Theorem \ref{thrm: uniqueness rect connected bounded} does not apply.

Regarding the quantitative uniqueness of Kantorovich potentials, interesting open questions include: Does some form of quantitative uniqueness result hold for unbounded target measures with common costs, such as $p$-costs? Is the exponent in Theorem \ref{thrm: c 1 alpha cost Linfty bound} optimal?

\subsection{Structure of the paper}

\begin{itemize}
    \item In Section \ref{ch: Kantorovich potentials}, we fix some conventions regarding Kantorovich potentials and clarify the precise sense in which our statements should be understood. We also recall the main results and tools from optimal transport that we will use throughout the paper.
    \item In Section \ref{ch: qual uniqueness}, we prove the two qualitative uniqueness results, Theorems \ref{thrm: uniqueness rect connected bounded} and \ref{thrm: uniqueness graph continuous case}.  We also provide two pedagogical examples regarding the (lack of) differentiability of Kantorovich potentials, and the (lack of) relation between primal and dual uniqueness.
    \item In Section \ref{ch: geometry of optimal set}, we present the half-space-based characterisation of $\Phi_c(\rho, \mu)$ for general cost, and develop the main machinery from which we will derive diameter bounds.
    \item In Section \ref{ch: quant uniqueness}, we prove the quantitative theorems, Theorems \ref{thrm: c 1 alpha cost Linfty bound}, \ref{thrm: quant L^p moment bounds}, \ref{thrm: quantitative uniqueness for measures on uniform grids} and \ref{thrm: quant stochastic}. We also provide an example which demonstrates that, in general, we cannot have quantitative uniqueness with an exponent better than $1$.
\end{itemize}

\section*{Acknowledgements}
\noindent This work was supported by a public grant from the Fondation Mathématique Jacques Hadamard. The author would like to thank his PhD supervisors Michael Goldman and Cyril Letrouit, for many stimulating discussions and suggestions throughout the preparation of this work.

\section{Kantorovich potentials}
\label{ch: Kantorovich potentials}

\subsection{Uniqueness in which space?}

Kantorovich potentials are often taken as either elements of $L^1(\rho) \times L^1(\mu)$, or $\mathcal{C}^0(\X) \times \mathcal{C}^0(\Y)$, for some $\X \supseteq \spt \rho$ and $\Y \supseteq \spt \mu$. In each of these spaces, different hypotheses on $\rho$, $ \mu$, and $c$ are required for strong duality, existence, and uniqueness for the dual problem \eqref{intro: c dual problem}. In this paper, we consider Kantorovich potentials as elements of $L^1$ in the sense of measurability and integrability. We will still work with pointwise defined functions rather than equivalence classes, since, depending on the choice of representatives defined on $\R^d \times \R^d$, a given pair $(\phi, \psi) \in L^1(\rho) \times L^1(\mu)$ may or may not satisfy the dual constraint \eqref{intro: dual constraint} on all $\R^d \times \R^d$. To guarantee that solutions exist to both the primal and $L^1$ dual problems, we make the following assumption on the cost, see \cite[Theorem 5.10]{villani2008optimal}.
\begin{assumption}[Cost regularity for primal and dual existence]
\hfill
\label{ass: existence hypotheses}
    \begin{enumerate}[label=(\roman*)]
        \item $c$ is lower semi-continuous.
        \item For all $(x, y) \in \R^d \times \R^d$, $|c(x, y)| \leq a(x) + b(y)$ for some real valued upper semi-continuous functions  $a \in L^1(\rho)$ and $ b \in L^1(\mu)$.
    \end{enumerate}
\end{assumption}
All our qualitative uniqueness results are thus statements of uniqueness $\rho$ and $\mu$ almost everywhere. In general, for $\X \supseteq \spt \rho$, there is no natural injection $\mathcal{C}^0(\X) \subseteq L^1(\rho)$, since the values of two $\phi_0, \phi_1 \in \mathcal{C}^0(\X)$ outside of $\spt \rho$ could be different, whilst belonging to the same $L^1(\rho)$ equivalence class. We do however have the injection $\mathcal{C}^0(\spt \rho) \subseteq L^1(\rho)$. Hence, our uniqueness results in $L^1$ do not, in general, imply uniqueness in $\mathcal{C}^0(\X)$ unless $\X = \spt \rho$.

\subsection{$c$-concavity, choosing good representatives}
In this subsection, we discuss why to prove uniqueness over all $L^1(\rho) \times L^1(\mu)$, it suffices to prove uniqueness over the smaller, more regular class of $c$-concave optimisers. Let $\overline \R = \R \cup \{- \infty\}$. Given $\phi, \psi : \R^d \to \overline \R$, we define their $(c, \spt \rho)$ and $(\overline c, \spt \mu)$ transforms
\begin{equation*}
    \phi^{c, \spt \rho}: \R^d \to \overline \R,\quad\quad \psi^{\overline c, \spt \mu} : \R^d \to \overline \R
\end{equation*}
by
\begin{equation}
\label{intro: c transf definition}
    \phi^{c, \spt \rho}(y) = \inf_{x \in \spt \rho} c(x, y) - \phi(x) ; \quad\quad \psi^{\overline c, \spt \mu}(x)= \inf_{y \in \spt \mu} c(x, y) - \psi(y).
\end{equation}
We refer to functions as $c$-concave if they are the $(c, \spt \rho)$ or $(\overline c, \spt \mu)$ transform of some function, and are not everywhere $- \infty$. We say $(\phi, \psi)$ are $(c, \spt \rho, \spt \mu)$ conjugates if they are each other's transform. Defined as such, a $c$-conjugate pair satisfies the dual admissibility condition \eqref{intro: dual constraint} on $\spt \rho \times \spt \mu$, but not necessarily on all $\R^d \times \R^d$.

Roughly speaking, all pairs of Kantorovich potentials are ``essentially $c$-concave", in the sense that given a pair of optimisers $(\phi, \psi)$ for \eqref{intro: c dual problem}, we can find a $(c, \spt \rho, \spt \mu)$-concave pair $(\tilde \phi, \tilde \psi)$ with $\phi = \tilde \phi$ $\rho$-a.e. and $\psi = \tilde \psi$ $\mu$-a.e. We formalise and prove this fact in Appendix \ref{ch: precisions on c concavity}, and clarify some of the subtleties with regards to defining the dual problem and $c$-transforms on different domains $\X \times \Y$ for sets $\X \subseteq \spt \rho$ and $\Y \supseteq \spt \mu$.

The important takeaway is that proving uniqueness $\rho$-a.e. and $\mu$-a.e. of optimisers to the dual problem \eqref{intro: c dual problem} is equivalent to proving uniqueness of $(c, \spt \rho, \spt\mu)$-concave optimisers. Throughout the rest of this paper, unless we say explicitly otherwise, we will always assume we are dealing with a $(c, \spt \rho, \spt \mu)$-concave pair of Kantorovich potentials.

\subsection{Duality, the compatibility condition}

We denote the set of optimisers for the primal problem \eqref{intro: c OT problem} by
\begin{equation*}
    \Gamma_c(\rho, \mu) :=\argmin_{\gamma \in \Pi(\rho, \mu)} \int_{\R^d \times \R^d} c(x, y) \di \gamma(x, y),
\end{equation*}
and set
\begin{equation*}
 \supp \Gamma_c(\rho, \mu) = \bigcup_{\gamma \in \Gamma_c(\rho, \mu)} \supp \gamma.
\end{equation*}
To ease notation, where clear from the context, we will sometimes just write $\spt \Gamma$. We emphasise that in general, this object depends on all three of $\rho, \mu$ and $c$. Kantorovich duality, \cite[Theorem 5.10]{villani2008optimal}, tells us that:
\begin{enumerate}[label=(\roman*)]
    \item Under Assumption \ref{ass: existence hypotheses}, the values of the primal and dual problems \eqref{intro: c OT problem} and \eqref{intro: c dual problem} are equal.
    \item For any $\gamma \in \Gamma_c(\rho, \mu)$ and any $(\phi, \psi)$ optimal for the dual problem \eqref{intro: c dual problem},
    \begin{equation}
    \label{eq: weak compatibility condition}
    \phi(x) + \psi(y) = c(x, y) \;\gamma-\text{a.e.}
    \end{equation}
    \item If for some admissible $\gamma$ and $(\phi, \psi)$, \eqref{eq: weak compatibility condition} holds, both are optimal for their respective problems.
\end{enumerate}
If we also assume that $c$ is continuous, then for any optimal $(c, \spt \rho , \spt \mu)$-concave pair $(\phi, \phi^{c, \spt \rho})$, we have the slightly stronger
\begin{equation}
    \label{intro: compatibility condition}
    \supp \gamma \subseteq \left\{ (x, y) \in \spt \rho \times \spt \mu : \phi(x) + \phi^{c, \spt \rho}(y) = c(x, y) \right\}.
\end{equation}
The set on the right-hand side of \eqref{intro: compatibility condition} is referred to as the graph of the $(c, \spt \rho, \spt \mu)$ superdifferential of $\phi$, denoted $\partial^c \phi$. The $c$-superdifferential can equivalently be characterised by $y \in \partial^c\phi(x)$ for some $x \in \spt \rho$ if and only if $y \in \spt \mu$ and
\begin{equation}
\label{eq: c superdiff}
    \forall z \in \spt \rho, \quad c(x, y) - \phi(x) \leq c(z, y) - \phi(z).
\end{equation}
Given $\rho, \mu, c$, we denote by $P_x(\Gamma)$, the projection $\proj_x$ of $\spt \Gamma_c(\rho, \mu)$ onto the first marginal
\begin{equation}
\label{eq: projection property closed operator}
   P_x(\Gamma): =  \proj_x(\spt \Gamma_c(\rho, \mu)).
\end{equation}
To ease notation, we do not write the dependence on $\rho, \mu$ and $c$ explicitly, but we remind the reader that in general, this set depends on all three objects. This set has full $\rho$ measure, and its closure is $\spt \rho$ so it is dense in $\spt \rho$. When $\spt \mu$ is bounded, $P_x(\Gamma) = \spt \rho$ due to the following lemma.
\begin{lemma}
\label{lem: projection is full support when mu bounded}
    Let $\rho$ and $\mu$ be probability measures on $\R^d$ with $\spt \mu$ bounded. Then for any $\gamma \in \Pi(\rho, \mu)$, $\proj_x(\spt \gamma)$ is closed and hence $\proj_x(\spt \gamma) = \spt \rho$.
\end{lemma}
\begin{proof}
    Take $x_n \in \proj_x(\spt \gamma)$ with $x_n \to x$, then there exists $y_n \in \spt \mu$ with $(x_n, y_n) \in \spt \gamma$. By compactness of $\spt \mu$, up to a subsequence $y_n \to y \in \spt \mu$ and hence $(x, y) \in \spt \gamma$ so that $x \in \proj_x(\spt \gamma)$.
\end{proof}

\subsection{Equivalence of pairwise and single-potential uniqueness}
\label{subsect: equivalence of pairwise and single potential unique}
The dual problem \eqref{intro: c dual problem} is an optimisation problem over $L^1(\rho) \times L^1(\mu)$, so one would expect uniqueness results to talk of the uniqueness of pairs of optimisers. We will not do so. This is because for any $c$, the $c$-transform is a $1$-Lipschitz operator with respect to the uniform norm;
\begin{equation*}
    \|\phi_0^{c, \spt \rho} - \phi_1^{c, \spt \rho}\|_{\mathcal{C}^0(\R^d)} \leq \|\phi_0 - \phi_1\|_{\mathcal{C}^0(\spt \rho)} \quad \forall \phi_0, \phi_1 : \R^d \to \R.
\end{equation*}
To see this, it suffices to observe
\begin{align*}
    \phi_0^{c, \spt \rho}(y) = \inf_{x \in \spt \rho} c(x, y) - \phi_0(x) \leq& \inf_{x \in \spt \rho} c(x, y) - \phi_1(x) + \|\phi_0 - \phi_1\|_{\mathcal{C}^0(\spt \rho)}\\
    =& \phi_1^{c, \spt \rho}(y) + \|\phi_0 - \phi_1\|_{\mathcal{C}^0(\spt \rho)},
\end{align*}
then reverse the roles played by each potential. It follows that to establish both qualitative uniqueness, as well as quantitative uniqueness in $L^\infty$ norm, it suffices to consider the $\rho$-side potentials.

\subsection{Equivalence between quadratic and bilinear costs}
\label{sect: equiv quadratic and bilinear}
The primal problems with quadratic cost $c_2(x, y) = \|x-y\|^2$ and bilinear cost $c_l(x, y) = - \langle x, y \rangle$ are equivalent. This is because for any $\gamma \in \Pi(\rho, \mu)$
\begin{equation}
   \int_{\R^d \times \R^d} \|x-y\|^2 \di \gamma =  \int_{\R^d} \|x\|^2 \di \rho + \int_{\R^d} \|y\|^2 \di \mu - 2 \int_{\R^d \times \R^d} \langle x, y \rangle \di \gamma,
\end{equation}
where the first two terms on the right are independent of the choice of $\gamma \in \Pi(\rho, \mu)$. By abuse of notation, when talking about $c_l$-potentials we instead will refer to their negatives, so that optimal $c_l$ potentials are convex functions, which are referred to as \textit{Brenier Potentials}. One can show that in the sense of Minkowski,
\begin{equation}
\label{eq: relation between quadratic and max corr potential sets}
    \Phi_{c_2}(\rho, \mu) = \|\cdot\|^2 - 2\Phi_{c_l}(\rho, \mu),
\end{equation}
see \cite[Proposition 2.1]{santambrogio2015optimal} for the idea. Thus, both qualitative and quantitative uniqueness results for the costs $c_2$ and $c_l$ are equivalent.

\section{Qualitative uniqueness}
\label{ch: qual uniqueness}

\subsection{Proof of Theorem \ref{thrm: uniqueness rect connected bounded}: Qualitative uniqueness with $\mathcal{C}^1$ cost}

To guarantee suitable regularity of $c$-concave functions, we will require the following assumptions on the cost.
\begin{assumption}[Cost regularity for uniqueness]
\label{ass: cost hypotheses uniqueness}
The functions $\{ x \mapsto c(x, y) \}_{y \in \spt \mu}$ are:
\begin{enumerate}[label=(\roman*)]
    \item Differentiable at each $x \in \R^d$.
    \item Locally equi-Lipschitz in $x$, i.e. for each compact $K\subset \R^d$,
    \begin{equation*}
        \|\nabla_x c\|_{\mathcal{C}^0(K \times \spt \mu)}< + \infty.
    \end{equation*}
\end{enumerate} 
\end{assumption}
\begin{proposition}
\label{prop: regularity of potentials given by cost}
    Let $\rho$ and $\mu$ be probability measures on $\R^d$, with $\spt \mu$ bounded. Suppose that $c \in \mathcal{C}^1(\R^d \times \R^d)$ and 
    \begin{equation}
       \forall (x, y) \in \R^d, \quad|c(x, y)| \leq a(x) + b(y) \label{eq: c bound by integrable functions in the chapter}
    \end{equation}
for some real-valued upper semi-continuous functions $a \in L^1(\rho)$ and $b \in L^1(\mu)$. Then $c$ satisfies both Assumptions \ref{ass: existence hypotheses} and \ref{ass: cost hypotheses uniqueness}. Consequently, for any pair of $(c, \spt \rho, \spt \mu)$-concave Kantorovich potentials $(\phi, \psi)$ on $\R^d \times \R^d$, $\phi$ is locally Lipschitz and everywhere finite.
\end{proposition}
\begin{proof}
    That $c \in \mathcal{C}^1$ satisfying \eqref{eq: c bound by integrable functions in the chapter} satisfies Assumption \ref{ass: existence hypotheses} is direct from our hypotheses. Since we assume that $\spt \mu$ is bounded, for any $K \subseteq \R^d$ compact we have $\|\nabla_x c\|_{\mathcal{C}^0(K \times \spt\mu)} <+ \infty$ so that Assumption \ref{ass: cost hypotheses uniqueness} holds also. The local Lipschitz property of potentials follows from standard envelope arguments; see \cite[Box 1.8]{santambrogio2015optimal}. It is not hard to show that if $\psi$ is optimal, then $\psi^{\overline{c}, \spt \mu} \in L^1(\rho)$, see Appendix \ref{ch: precisions on c concavity}. Thus, we cannot have $\psi^{\overline c, \spt \mu} \equiv - \infty$. Since $\psi^{\overline c, \spt \mu}$ is finite at at least one point, and is locally Lipschitz on $\R^d$, it is finite everywhere.
\end{proof}

\begin{proof}[Proof of Theorem \ref{thrm: uniqueness rect connected bounded}]
Let $\phi \in \Phi_c(\rho, \mu)$. Thanks to Proposition \ref{prop: selection principle for c concave representative}, we can assume without loss of generality that $\phi$ is $(\overline c, \spt \mu)$-concave so that $\phi$ is locally Lipschitz on $\spt \rho$ by Proposition \ref{prop: regularity of potentials given by cost}. Take $x', x \in \supp \rho$, and let
\begin{equation*}
    \omega : [0, 1] \to \spt \rho,\quad \quad \omega(0) = x',\; \omega(1) = x,
\end{equation*}
be a finite-length curve, taken with constant speed parametrisation so that it is Lipschitz. The function $\phi$ is Lipschitz on the compact set $\omega([0, 1]) \subseteq \spt \rho$, and so
 \begin{equation*}
     \tilde\phi : [0, 1] \to \R; \quad\tilde\phi(t) := \phi(\omega(t))
 \end{equation*}
is also Lipschitz as a composition of Lipschitz functions. By Rademacher's theorem, \cite[Chapter 3]{evans2025measure}, both $\tilde \phi$ and $\omega$ are differentiable $t$-a.e., so the intersection of these points has full Lebesgue measure on $[0, 1]$. Fix $t \in (0, 1)$ as such a differentiability point. Since $\spt \mu$ is bounded, by Lemma \ref{lem: projection is full support when mu bounded} there exists at least one $y \in \R^d$ such that $(\omega(t), y) \in \spt \Gamma$. By the compatibility condition \eqref{intro: compatibility condition} combined with a $c$-superdifferential inequality \eqref{eq: c superdiff} with $x= \omega(t)$, $z = \omega(t + h)$, we have
\begin{equation*}
    c(\omega(t), y) - \tilde\phi(t) \leq c(\omega(t+h), y) - \tilde\phi(t+h).
\end{equation*}
Rearranging, dividing by $h \neq 0$ and considering the limits $h \to 0^+, 0^-$ which exist by differentiability, we deduce
\begin{equation}
\label{eq: potential value along the curve}
    \dot{\tilde\phi}(t) = \langle \nabla_x c(\omega(t), y), \dot\omega(t)\rangle.
\end{equation}
Recall that $\phi \in \Phi_c(\rho, \mu)$ was arbitrary, and observe that the right-hand side of \eqref{eq: potential value along the curve} depends only on $\omega$, $t$, $c$ and $\Gamma_c(\rho, \mu)$. Thus, what we have shown is that the projection of the set 
\begin{equation*}
    \nabla_x c\Big(\omega(t), \partial^c \phi\big(\omega(t)\big)\Big)
\end{equation*}
in the direction $\dot\omega(t)$ is $[0, 1]$-Lebesgue a.e. a single value, and this value is the same for any $\phi \in \Phi_c(\rho, \mu)$. Hence given any two optimal potentials $\phi_0, \phi_1 \in \Phi_c(\rho, \mu)$, defining $\tilde \phi_i = \phi_i \circ \omega$, we have $\dot{\tilde \phi}_0(t) = \dot{\tilde \phi}_1(t)$ $t$-a.e. Thus, integrating along $[0, 1]$, we conclude that for all $x', x \in \spt \rho$,
\begin{equation*}
\phi_0(x) - \phi_0(x') = \phi_1(x) - \phi_1(x') \iff \phi_0(x) - \phi_1(x) = \phi_0(x') - \phi_1(x').
\end{equation*}
In other words, $\phi_0 - \phi_1$ is constant across $\spt \rho$, so that Kantorovich potentials are unique up to a constant in $L^1$ as required.
\end{proof}
\begin{remark}[Necessity of the hypothesis $\spt \mu$ bounded]
\label{rmk: spt mu bounded}
    We use the hypothesis that $\spt \mu$ is bounded on two separate occasions, both corresponding to essentially the same phenomena. The first occasion is to ensure that for $c \in \mathcal{C}^1$, Assumption \ref{ass: cost hypotheses uniqueness} holds, guaranteeing local Lipschitz regularity of $\rho$-side Kantorovich potentials. The second occasion is to guarantee that $\proj_x(\spt \gamma) = \spt \rho$ for any $\gamma \in \Pi(\rho, \mu)$.
    
    Both speak to the question of for which $x \in \spt \rho$, mass can be locally ``sent to infinity" by some $\gamma \in \Gamma_c(\rho, \mu)$, with this being related to the gradient of optimal potentials via the compatibility condition \eqref{intro: compatibility condition}. The points $x \in \spt \rho \setminus P_x(\Gamma)$ are those such that when $x_n \to x$ with $(x_n, y_n) \in \spt \Gamma$, we have $\|y_n\| \to \infty$. For some costs, including $p$-costs for $p >1$, in \cite[Section 5]{staudt2025uniqueness} the authors prove that for any $K \subseteq \Int \spt \rho$ compactly contained, the `` image" of all points under any optimal plan is locally uniformly bounded, restricting points where mass is being ``sent locally to infinity" to $\partial \spt \rho$. This is sufficient for them to establish uniqueness with $\spt \mu$ unbounded, under the closure of a connected open set hypothesis for $\spt \rho$. This argument does not directly apply if $\spt \rho$ is only rectifiably connected.
    
    As a consequence of the arguments of \cite{staudt2025uniqueness}  combined with ours for Theorem \ref{thrm: uniqueness rect connected bounded}, one can deduce uniqueness for the same costs as theirs with $\spt \mu$ unbounded and $\spt \rho$ connected by arcs which stay inside the interior of the convex hull of $\spt \rho$, and for which $\rho$ gives no mass to $\partial \conv \spt \rho$. Alternatively, using our quantitative arguments from Section \ref{ch: geometry of optimal set}, one can deduce uniqueness for $\rho, \mu$ unbounded probability measures on $\R^d$ if $\spt \rho$ is connected by continuous arcs such that $\dot \omega \in BV$, and $c \in \mathcal{C}^{1, \alpha}(\R^d \times \R^d)$. Due to the requirement that $c$ is globally $\mathcal{C}^{1, \alpha}$, this does not apply to any $p$-costs. Neither of these responses provides a complete answer to Open Problem \ref{problem: unbounded target}.
\end{remark}
\subsection{Examples: Non-differentiability, relation to primal uniqueness}

\begin{example}[Non differentiability of potentials everywhere inside the support]
The proof only used Rademacher's theorem to deduce that the one-dimensional function $\tilde \phi : [0, 1] \to \R$  was differentiable a.e. with respect to Lebesgue measure on $[0, 1]$. This does not mean that $\phi : \R^d \to \R$ is differentiable $\rho$ a.e. Theorem \ref{thrm: uniqueness rect connected bounded} still applies in cases where potentials are actually nowhere differentiable on $\spt \rho$, as demonstrated by the following example. On $\R^2$ with bilinear cost $c(x, y) = - \langle x , y \rangle$, let
    \begin{equation*}
        \rho = \mathcal{H}^1|_{\{0\} \times [0, 1]} \quad \text{ and } \quad \mu = \frac{1}{2}\mathcal{H}^1|_{\{-1\} \times [0, 1]} + \frac{1}{2}\mathcal{H}^1|_{\{1\} \times [0, 1]}.
    \end{equation*}
Then the unique optimal plan splits $\rho$'s mass horizontally at each point in $\spt \rho$, see Figure \ref{fig:uniqueness of gradient projection}.
\begin{figure}[ht]
        \centering
\begin{tikzpicture}[>=Latex,scale=3]
  \tikzset{
    rho/.style   ={line width=1.5pt, red!70, line cap=round},
    mu/.style    ={line width=1.5pt, blue!70, line cap=round},
    plan/.style  ={->, line width=0.9pt, gray!70, line cap=round},
  }

  \draw[very thin] (-1.15,0) -- (1.15,0);
  \draw[very thin] (0,-0.02) -- (0,1.08);

  \fill[red!10] (0,0) rectangle (-1,1);
  \fill[red!10] (0,0) rectangle ( 1,1);

  \foreach \y in {0.20,0.50,0.80}{
    \draw[plan] (-0.02,\y) -- (-0.78,\y);
    \draw[plan] ( 0.02,\y) -- ( 0.78,\y);
  }

  \draw[mu]  (-1,0) -- (-1,1);
  \draw[mu]  ( 1,0) -- ( 1,1);
  \draw[rho] ( 0,0) -- ( 0,1);
\end{tikzpicture}
\caption{The optimal transport plan between $\rho$ (middle) and $\mu$ (left and right).}
\label{fig:uniqueness of gradient projection}
\end{figure}
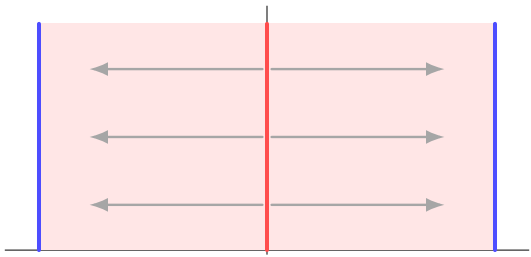
Thus, all Kantorovich potentials are nowhere differentiable on $\spt \rho$, their subgradients must contain at least two points by the compatibility condition \eqref{intro: compatibility condition}. Nonetheless, the hypotheses of Theorem \ref{thrm: uniqueness rect connected bounded} hold, so we have dual uniqueness. The proof implicitly deduces that the projection of $\partial\phi$ in the direction $\omega'(t)$ is a.e. unique.
\end{example}
\begin{example}[Independence of uniqueness for primal and dual]
Uniqueness of primal optimisers does not imply uniqueness of dual optimisers, nor vice versa. This is well known in the discrete setting, in the language of linear programming. In \cite[Remark 3.5]{letrouit2024gluing}, the authors provide an example of a non-unique dual for quadratic cost, while the primal problem has a unique solution.
Conversely, on $\R^2$ with quadratic cost, let
    \begin{equation*}
        \rho = \mathcal{H}^1|_{\{0\} \times [-1, 1]} \quad \text{ and } \quad \mu = \mathcal{H}^1|_{[1, 3] \times \{0\}}.
    \end{equation*}
    Here, every transport plan $\gamma \in \Pi(\rho, \mu)$ is optimal for quadratic cost. Yet Theorem \ref{thrm: uniqueness rect connected bounded} implies dual uniqueness.
\end{example}

\subsection{Proof of Theorem \ref{thrm: uniqueness graph continuous case}: Graph-based uniqueness}
Let $(\phi_0, \psi_0), (\phi_1, \psi_1)$ be two pairs of $(c, \spt \rho, \spt \mu)$-concave optimisers for the dual problem \eqref{intro: c dual problem}. The proof of Theorem \ref{thrm: uniqueness rect connected bounded} directly establishes that the restrictions $\phi_l|_{A_i}$ and $\psi_l|_{B_j}$ are unique up to a constant, and hence there exist $\{a_i\}_{i \in I}, \{b_j\}_{j \in J} \subseteq \R$ such that
\begin{equation*}
    (\phi_0 - \phi_1)|_{A_i} = a_i \quad \text{ and } \quad (\psi_0 - \psi_1)|_{B_j} = b_j.
\end{equation*}
We want to show that all the $a_i$ are equal to some $a \in \R$ and hence $\phi_0-\phi_1$ is constant across all $\spt \rho$.
Fix some $i, i' \in I$, let $x \in A_i$ and $x' \in A_{i'}$. Let $i=i_0, j_0, i_1, j_1,..., j_{n-1}, i_n = i'$ be a path between $A_i$ and $A_{i'}$ in the graph $G$ as defined in the theorem statement. For any $A \times B \subset \R^d$ measurable with $\gamma(A \times B) >0$, we have $(A \times B) \cap \spt \gamma \neq \emptyset$, since if $A \times B$ and $\spt \gamma$ were disjoint,
\begin{equation*}
    \gamma(A \times B) + \gamma(\spt \gamma) = \gamma(\spt \gamma \cup (A \times B))
\end{equation*}
which would imply $\gamma(A \times B) = 0$. It follows that there exist points
\begin{equation*}
    x=x_{0}, \hat x_{0}, y_{0}, \hat y_{0}, x_{1}, \hat x_{1},..., y_{{n-1}}, \hat y_{{n-1}}, x_{n}, \hat x_{n} = x'
\end{equation*}
with $x_{k}, \hat x_{k} \in A_{i_k}$ and $y_{k}, \hat y_{k} \in B_{i_k}$ such that
\begin{equation*}
    (\hat x_{k}, y_{k}), (x_{k+1}, \hat y_{k}) \in \spt \Gamma \quad \forall k = 0,..., n-1.
\end{equation*}
By the compatibility condition \eqref{intro: compatibility condition}, for $l = 0, 1$,
\begin{equation}
    \phi_l(\hat x_{k}) + \psi_l(y_k) = c(\hat{x}_k , y_k) \quad \text{ and }\quad \phi_l(x_{k+1}) + \psi_l(\hat y_k) = c(x_{k+1}, \hat y_k),
\end{equation}
so that
\begin{equation*}
    \phi_0(\hat x_{k}) - \phi_1(\hat x_{k}) =   \psi_1(y_k) -  \psi_0(y_k) \quad \text{ and } \quad \phi_0(x_{k+1}) - \phi_1(x_{k+1}) = \psi_1(\hat y_k) - \psi_0(\hat y_k).
\end{equation*}
It follows that
\begin{align*}
    a_i &= \phi_0(x_0) - \phi_1(x_0) = \phi_0(\hat x_0) - \phi_1(\hat x_0) = \psi_1(y_0) - \psi_0(y_0)\\ &= \psi_1(\hat y_0) - \psi_0( \hat y_0) = \phi_0(x_1) - \phi_1(x_1) = \cdots = \phi_0(\hat x_n) - \phi_1(\hat x_n) = a_{i'},
\end{align*}
and so $a_i = a_{i'} = -b_{j} = -b_{j'}$ for all $i, j \in I \times J$, which is precisely what we wanted to prove. \qedwhite
\begin{remark}
    In \cite{staudt2025uniqueness}, the authors establish uniqueness under the assumption that there are no subsets of indices $I' \subseteq I$ and $J' \subseteq J$ such that
\begin{equation}
    \label{intro:eq: german nondegeneracy}
    0<\rho( \cup_{i \in I'} A_i) = \mu( \cup_{j \in J'} B_j)<1. 
\end{equation}
This hypothesis is much stronger than ours. Under this assumption, for any $\gamma \in \Pi(\rho, \mu)$, the support graph $G_\gamma$ defined in the sense of Theorem \ref{thrm: uniqueness graph continuous case} is connected.
\end{remark}

\section{Geometry of the set of Kantorovich potentials}
\label{ch: geometry of optimal set}
\subsection{Characterisation of $\Phi_c(\rho,\mu)$ as an intersection of half spaces}

The dual problem \eqref{intro: c dual problem} has a continuous, linear objective and a closed, convex constraint. Hence, the set of optimisers is closed and convex. It is well known that such a set can be represented as the intersection of closed half-spaces; see, for example, \cite[Theorem 1.48]{barbu2012convexity}. In this section, we construct an explicit family of half-spaces for which this is the case. Our characterisation generalises some other implicit characterisations of the set of fully discrete dual optimisers given in \cite[Section 3.3]{acciaio2025characterization} and \cite[Theorem 3.5]{nutz2025quadratically}.

We repeat the calculations given in the introduction for a general cost $c$. Let $\{(x_i, y_i)\}_{i=0}^n \subseteq \spt \Gamma$ and take any $\phi \in \Phi_{c}(\rho, \mu)$. The compatibility condition \eqref{intro: compatibility condition} implies $y_i \in \partial^c \phi(x_i)$ and hence $c$-supergradient inequalities \eqref{eq: c superdiff} imply that for $i = 0,..., n-1$,
\begin{equation*}
    c(x_i, y_i) - \phi(x_i) \leq c(x_{i+1}, y_i) - \phi(x_{i+1}) \quad \text{ and } \quad  c(x_{i+1}, y_{i+1}) - \phi(x_{i+1}) \leq c(x_i, y_{i+1}) - \phi(x_i).
\end{equation*}
These imply upper and lower bounds on the relative values $\phi(x_{i+1}) - \phi(x_{i})$:
\begin{equation}
\label{eq: upper and lower bound c subsdiff inequalities}
c(x_{i+1}, y_{i+1}) - c(x_i, y_{i+1}) \leq  \phi(x_{i+1})- \phi(x_{i}) \leq c(x_{i+1}, y_i) - c(x_i, y_i).
\end{equation}
Summing, we deduce
\begin{equation}
\label{eq: c gradient upper and lower bounds}
    \sum_{i=0}^{n-1} \big[c(x_{i+1}, y_{i+1}) - c(x_i, y_{i+1})\big] \leq \phi(x_n) - \phi(x_0) \leq \sum_{i=0}^{n-1}\big[c(x_{i+1}, y_i) - c(x_i, y_i)\big].
\end{equation}
We recall the definition \eqref{eq: projection property closed operator} of $P_x(\Gamma) \subseteq \spt \rho$, which has full $\rho$ mass and whose closure is $\spt \rho$, with equality when $\spt \mu$ is bounded. Taking the infimum in \eqref{eq: c gradient upper and lower bounds} over all finite chains $\{(x_i, y_i)\}_{i=0}^n \subset \spt \Gamma$, we define
\begin{align*}
    &\lambda_{\rho, \mu, c} : P_x(\Gamma) \times P_x(\Gamma) \to \R\\
    &(x, x') \mapsto \inf_{\substack{\{(x_i, y_i)\}_{i=0}^n \subseteq \spt \Gamma \\ x_0 = x',\; x_n = x}} \sum_{i=0}^{n-1}\big[c(x_{i+1}, y_i) - c(x_i, y_i)\big].
\end{align*}
Then for any optimal potential $\phi \in \Phi_c(\rho, \mu)$,
\begin{equation*}
    \forall x, x' \in P_x(\Gamma), \quad -\lambda_{\rho, \mu, c}(x', x) \leq \phi(x) - \phi(x') \leq \lambda_{\rho, \mu, c}(x, x').
\end{equation*}
Taking a supremum on the left-hand side of \eqref{eq: c gradient upper and lower bounds} over points $\{(x_i, y_i)\}_{i=0}^n \subset \spt \Gamma$, we can recover an alternative definition of $-\lambda_{\rho, \mu, c}(x', x)$ by symmetry of our calculation. In particular, we also have
\begin{equation}
\label{eq: lambda alt definition}
    \lambda_{\rho, \mu, c}(x', x) = \inf_{\substack{\{(x_i, y_i)\}_{i=0}^n \subseteq \spt \Gamma \\ x_0 = x',\;  x_n = x}} \sum_{i=0}^{n-1} \big[c(x_i, y_{i+1}) - c(x_{i+1}, y_{i+1})\big].
\end{equation}
(In general, $\lambda$ is not symmetric in its arguments.) To ease notation, for the proofs, we will omit the dependence on $\rho, \mu$ and $c$.
\begin{proposition}[$\Phi$ as an intersection of half-spaces]
\label{prop: char of the optimal set}
Let $\rho$ and $\mu$ be probability measures on $\R^d$, suppose that $c$ is continuous and satisfies Assumption \ref{ass: existence hypotheses}. Then
    \begin{equation}
    \label{eq: characterisation of the dual set}
        \Phi_{c}(\rho, \mu) = \left\{ \phi \in L^1(\rho): \phi(x) - \phi(x') \leq \lambda_{\rho, \mu, c}(x, x') \quad \forall x, x' \in P_x(\Gamma) \right\}.
    \end{equation}
\end{proposition}
\begin{proof}
    For any potential $\phi \in \Phi_{c}(\rho, \mu)$, the inclusion $\subseteq$ is a direct consequence of the above calculations. For the converse, assume $\phi: P_x(\Gamma) \to \overline \R$ belongs to the right-hand side of \eqref{eq: characterisation of the dual set}. We change the values of $\phi$ to be $-\infty$ on the $\rho$-null set $\R^d \setminus P_x(\Gamma)$ (which we are allowed to do since $P_x(\Gamma)$ has full $\rho$ mass.) Then the potentials $(\phi, \phi^{c, \spt \rho})$ are optimal for the dual problem \eqref{intro: c dual problem}. To see this, consider all possible pairs $(x_0, y_0), (x_1, y_1) \in \spt \Gamma$. By definition of $\lambda$, we have
    \begin{equation}
       \forall x_1 \in P_x(\Gamma), \quad c(x_0, y_0) \leq c(x_1, y_0) - \phi(x_1) + \phi(x_0),
    \end{equation}
    and consequently after passing the infimum over $x_1 \in \spt \rho$,
    \begin{equation}
    \label{eq: compatability of rho transform on the support of Gamma}
         \forall (x_0,y_0) \in \spt \Gamma, \quad c(x_0, y_0) \leq \phi(x_0) + \phi^{c, \spt \rho}(y_0).
    \end{equation}
    By definition of the transform, $(\phi, \phi^{c, \spt \rho})$ is admissible for the dual problem on $\spt \rho \times \spt \mu$, and by \eqref{eq: compatability of rho transform on the support of Gamma}, the pair satisfy the compatibility condition \eqref{intro: compatibility condition}. They are thus an optimal pair for \eqref{intro: c dual problem} so that $\phi \in \Phi_c(\rho, \mu)$ as required.
\end{proof}
\begin{remark}[Connections to graph based uniqueness]
Suppose that $\{(x_i, y_i)\}_{i=0}^n\subseteq \spt \Gamma$, correspond to a path in the graph $\spt \Gamma$ from $x_0 = x'$, $x_n = x$, in the sense that we also have
\begin{equation*}
    (x_{i+1}, y_i) \in \spt \Gamma \quad \forall i =0,...,n-1.
\end{equation*}
Then both inequalities in \eqref{eq: upper and lower bound c subsdiff inequalities} are actually equalities, and thus so are those in \eqref{eq: c gradient upper and lower bounds}. We can deduce that $-\lambda(x',x) = \lambda(x, x')$, i.e. the interval of admissibility is a singleton. This recovers the graph-based uniqueness, Theorem \ref{thrm: uniqueness graph continuous case} (in the fully discrete case).  
\end{remark}

\subsection{Diameter bounds}

To obtain bounds on the diameter, it suffices to prove that the intervals of admissibility for $\phi(x') - \phi(x)$ are small for every pair $x, x' \in P_x(\Gamma)$.
\begin{proposition}
\label{prop: the direct bound expression}
Let $\rho$ and $\mu$ be probability measures on $\R^d$, suppose that $c$ is continuous and satisfies Assumption \ref{ass: existence hypotheses}. Fix some $x' \in P_x(\Gamma)$, then for any $\phi_0, \phi_1 \in \Phi_c(\rho, \mu)$ with $\phi_0(x') = \phi_1(x')$ and any $x \in P_x(\Gamma)$, we have
    \begin{equation}
    \label{eq: the direct bound expression}
        |\phi_1(x) - \phi_0(x)| \leq \inf_{\substack{\{(x_i, y_i)\}_{i=0}^n \subseteq \spt \Gamma \\ x_0 = x',\;  x_n = x}} \sum_{i=0}^{n-1} \big[c(x_{i+1}, y_i) - c(x_i, y_i) + c(x_i, y_{i+1}) - c(x_{i+1}, y_{i+1})\big].
    \end{equation}
\end{proposition}
\begin{proof}
Let $\phi_0, \phi_1 \in \Phi_c(\rho, \mu)$. From Proposition \ref{prop: char of the optimal set}, we have 
\begin{equation*}
    \phi_0(x) - \phi_0(x'),\, \phi_1(x) - \phi_1(x') \in \big[-\lambda(x', x), \lambda(x, x')\big],
\end{equation*}
and hence
    \begin{align*}
        |\phi_1(&x) - \phi_0(x)| = |\phi_1(x) - \phi_1(x') -(\phi_0(x) - \phi_0(x'))| \leq \lambda(x, x') + \lambda(x', x) \\
        =& \inf_{\substack{\{(x_i, y_i)\}_{i=0}^n \subseteq \spt \Gamma\\ x_0 = x',\;  x_n = x}} \sum_{i=0}^{n-1} \big[c(x_{i+1}, y_i) - c(x_i, y_i)\big] + \inf_{\substack{\{(x_i, y_i)\}_{i=0}^n \subseteq \spt \Gamma \\ x_0 = x',\;  x_n = x}} \sum_{i=0}^{n-1} \big[c(x_i, y_{i+1}) - c(x_{i+1}, y_{i+1})\big]\\
        \leq& \inf_{\substack{\{(x_i, y_i)\}_{i=0}^n \subseteq \spt \Gamma \\ x_0 = x',\;  x_n = x}} \sum_{i=0}^{n-1} \big[c(x_{i+1}, y_i) -c(x_i, y_i) + c(x_i, y_{i+1}) - c(x_{i+1}, y_{i+1})\big]
    \end{align*}
    as required, where we used the alternative characterisation \eqref{eq: lambda alt definition} for $\lambda(x', x)$.
\end{proof}
The basis of all our quantitative results is Lemma \ref{lem: michael curve approx} below, which is based on the following intuition. Ignoring boundary terms and re-indexing the third and fourth terms of the sum \eqref{eq: the direct bound expression}, one views the terms as
\begin{equation}
\label{eq: second difference of the cost term}
    c(x_{i+1}, y_i) - c(x_i, y_i) - \big( c(x_i, y_i) - c(x_{i-1}, y_i)\big).
\end{equation}
Take a continuous curve $\omega : [0, 1] \to \R^d$, and choose $x_i = \omega(i/n)$ for $i = 0,...,n$ for large $n \in \mathbb{N}$. Then the expression \eqref{eq: second difference of the cost term} is a second difference of the function $ t \mapsto c(\omega(t), y_i)$. If $c$ and $\omega$ are both of class $\mathcal{C}^2$, these second differences are of order $O(1/n^2)$ and consequently the sum of $n$ such terms is $O(1/n)$. This idea extends more generally to $c$ of class $\mathcal{C}^{1, \alpha}$ and curves $\omega$ such that $\dot \omega \in BV([0, 1]; \R^d)$.
\begin{lemma}
\label{lem: michael curve approx}
Let $\X, \Y \subseteq \R^d$, and let $c \in \mathcal{C}^{1, \alpha}(\X \times \Y)$ for some $\alpha \in (0, 1]$. Let $\{x_i\}_{i=0}^n \subset \X$ and $\{y_i\}_{i=0}^n \subset \Y$. Let $\omega : [0, 1] \to \X$ with $\dot \omega \in BV([0, 1]; \R^d)$, and set $t_i = i/n$ for $i = 0,...,n$. Then
    \begin{align}
        \sum_{i=0}^{n-1} \big[c(x_{i+1}, &y_i) - c(x_i, y_i) + c(x_i, y_{i+1}) - c(x_{i+1}, y_{i+1})\big]\label{eq: 0th line?}\\
        \leq &2\| \nabla_x c\|_{\mathcal{C}^0(\X \times \Y)}\left(\|x_0 - \omega(t_0)\| + \|x_n - \omega(t_n)\| + 2 \sum_{i=1}^{n-1}\|x_i - \omega(t_i)\|\right) \label{eq: 1st line?}\\
        & + \frac{2\|\nabla_x c\|_{\mathcal{C}^0(\X \times \Y)} \|\dot\omega\|_{L^\infty}}{n} \label{eq: 2nd line?}\\
        &+ \frac{\|\nabla_x c\|_{\mathcal{C}^0(\X \times \Y)} |\ddot \omega |([0, 1])}{n}\label{eq: 3rd line?}\\
        &+\frac{[ \nabla_x c]_{\X \times \Y, \alpha}\|\dot \omega\|_{L^\infty}^{\alpha + 1}}{n^{\alpha}}\label{eq: 4th line?}.
    \end{align}
\end{lemma}
\begin{proof}
We rewrite \eqref{eq: 0th line?} as two sums, re-indexing the third and fourth terms:
\begin{align*}
    \sum_{i=0}^{n-1} \big[c(x_{i+1}, &y_i) - c(x_i, y_i) + c(x_i, y_{i+1}) - c(x_{i+1}, y_{i+1})\big] \\
    =&\sum_{i=0}^{n-1} \big[c(x_{i+1}, y_i) - c(x_i, y_i)\big] + \sum_{i=1}^{n} \big[c(x_{i-1}, y_{i}) - c(x_{i}, y_{i})\big]\\
    =& c(x_1, y_0) - c(x_0, y_0) + c(x_{n-1}, y_n) - c(x_n, y_n)\\
    &+\sum_{i=1}^{n-1} \big[c(x_{i+1},y_i) + c(x_{i-1}, y_i) - 2 c(x_i, y_i)\big].
\end{align*}
For each $i=0,...,n$ and any $y \in \Y$, we have
\begin{equation*}
    |c(x_i, y) - c(\omega(t_i), y)| \leq \|\nabla_x c\|_{\mathcal{C}^0(\X \times \Y)} \| x_i - \omega(t_i)\|,
\end{equation*}
so that approximating each $x_i$ by $\omega(t_i)$ gives the first line \eqref{eq: 1st line?}, and we are left to bound
\begin{align}
    & c(\omega(t_1), y_0) - c(\omega(t_0), y_0) + c(\omega(t_{n-1}), y_n) - c(\omega(t_n), y_n)\label{eq: curve bound c first line}\\
    &+\sum_{i=1}^{n-1} \big[c(\omega(t_{i+1}),y_i) + c(\omega(t_{i-1}), y_i) - 2 c(\omega(t_i), y_i)\big]. \label{eq: curve bound c second line}
\end{align}
The first two terms in \eqref{eq: curve bound c first line} are bounded by
\begin{equation*}
    |c(\omega(t_1), y_0) - c(\omega(t_0), y_0)| \leq \|\nabla_x c\|_{\mathcal{C}^0(\X \times \Y)} \|\omega(t_1) - \omega(t_0)\| \leq \frac{\|\nabla_x c\|_{\mathcal{C}^0(\X \times \Y)} \|\dot\omega\|_{L^\infty}}{n},
\end{equation*}
with the second pair of terms treated identically, giving \eqref{eq: 2nd line?}. We now concentrate on \eqref{eq: curve bound c second line}. For any $i = 1,...,n-1$, we have
\begin{align}
    c(\omega(t_{i+1}),y_i)& + c(\omega(t_{i-1}), y_i) - 2c(\omega(t_i), y_i)\nonumber\\
    =\Big(c(\omega(t_{i+1}),y_i)& - c(\omega(t_i), y_i) \Big) - \Big(c(\omega(t_i), y_i) - c(\omega(t_{i-1}), y_i)\Big)\nonumber\\
    =& \int_{t_i}^{t_{i+1}} \Big\langle\nabla_x c(\omega(t), y_i),\,  \dot \omega(t) \Big\rangle \di t - \int_{t_{i-1}}^{t_{i}} \Big\langle\nabla_x c(\omega(t), y_i),\,  \dot \omega(t) \Big\rangle \di t\nonumber\\
    =& \int_{t_i}^{t_{i+1}} \Big\langle\nabla_x c(\omega(t), y_i),\,  \dot \omega(t) \Big\rangle - \Big\langle \nabla_x c(\omega(t-1/n), y_i),\, \dot \omega(t-1/n) \Big\rangle \di t\nonumber\\
    =& \int_{t_i}^{t_{i+1}} \Big\langle\nabla_x c(\omega(t), y_i),\,  \dot \omega(t) - \dot \omega(t-1/n) \Big\rangle \di t \label{eq: first cost term to bound}\\
    &+ \int_{t_i}^{t_{i+1}} \Big\langle\nabla_x c(\omega(t), y_i) - \nabla_x c(\omega(t-1/n), y_i),\,\dot \omega(t-1/n) \Big\rangle \di t. \label{eq: second cost term to bound}
\end{align}
Summing the first term \eqref{eq: first cost term to bound}, we have
\begin{align*}
    \sum_{i=1}^{n-1} \int_{t_i}^{t_{i+1}} \Big\langle\nabla_x c(\omega(t), y_i),&\,  \dot \omega(t) - \dot \omega(t-1/n) \Big\rangle \di t\\ \leq& \|\nabla_x c\|_{\mathcal{C}^0(\X \times \Y)}\sum_{i=1}^{n-1} \int_{t_i}^{t_{i+1}} \| \dot \omega(t) - \dot \omega(t-1/n) \| \di t\\
    \leq& \|\nabla_x c\|_{\mathcal{C}^0(\X \times \Y)} \sum_{i=1}^{n-1} \int_{0}^{1/n} |\ddot \omega|\Big((t + t_{i-1}, t + t_i]\Big) \di t\\
    =& \|\nabla_x c\|_{\mathcal{C}^0(\X \times \Y)} \int_0^{1/n} |\ddot \omega|\Big((t, t + (n-1)/n]\Big) \di t\\
    \leq& \frac{\|\nabla_x c\|_{\mathcal{C}^0(\X \times \Y)} |\ddot \omega |([0, 1])}{n},
\end{align*}
which gives \eqref{eq: 3rd line?}. (Here we used that without loss of generality we can assume we are dealing with a right continuous representative of $\dot \omega$, so that for any $0\leq s < t \leq 1$ we have $\dot \omega(t) - \dot \omega(s) = \ddot \omega ( (s, t])$.) Finally, we treat \eqref{eq: second cost term to bound}. For any $y \in \Y$ and $t \in [1/n, 1]$, we have
\begin{align*}
 \|\nabla_x c(\omega(t), y) - \nabla_x c(\omega(t-1/n), y)\| \leq& [\nabla_x c]_{\X \times \Y, \alpha} \|\omega(t) - \omega(t-1/n)\|^{\alpha}\\ \leq& [\nabla_x c]_{\X \times \Y, \alpha} \left(\frac{\|\dot \omega\|_{L^\infty}}{n}\right)^\alpha
\end{align*}
and hence
\begin{align*}
       \sum_{i=1}^{n-1} \int_{t_i}^{t_{i+1}} &\Big\langle\nabla_x c(\omega(t), y_i) - \nabla_x c(\omega(t-1/n), y_i),\dot \omega(t-1/n) \Big\rangle \di t\\
        \leq& [\nabla_x c]_{\X \times \Y, \alpha}\sum_{i=1}^{n-1} \int_{t_i}^{t_{i+1}} \frac{\|\dot \omega\|_{L^\infty}^{\alpha + 1}}{n^{\alpha}} \di t \leq \frac{[\nabla_x c]_{\X \times \Y, \alpha}\|\dot \omega\|_{L^\infty}^{\alpha + 1}}{n^{\alpha}}.
\end{align*}
This gives the final term \eqref{eq: 4th line?} as required.
\end{proof}
\begin{remark}
\label{rmk: invariance under cost translations}
    The terms appearing in Proposition \ref{prop: the direct bound expression} are invariant under the replacement of the cost $c$ by $\tilde c(x, y) = c(x, y) + a(x) + b(y)$ for any $a \in L^1(\rho)$ and $b \in L^1(\mu)$ continuous. (This is the same phenomenon as in Section \ref{sect: equiv quadratic and bilinear}.) In the proof of Lemma \ref{lem: michael curve approx}, we only ever use the differentiability in $x$ of $c$, and equi-Hölder continuity of the $y \in \Y$ parametrised functions
\begin{equation*}
	c_y : \X \to \R; \quad c_y(x) = c(x, y),
\end{equation*}
so $[\nabla_x c]_{\X \times \Y, \alpha}$ in \eqref{eq: 4th line?} can be replaced by the asymmetric quantity
    \begin{equation}
\label{eq: asymetric Holder definition}
        \sup_{y \in \Y} [\nabla c_y]_{\X, \alpha}.
    \end{equation}
For quadratic cost $c(x, y) = \|x-y\|^2$, we can choose $a(x) = - \| x\|^2$ so that
    \begin{equation*}
        \nabla_x \tilde c = -2y
    \end{equation*}
so that \eqref{eq: asymetric Holder definition} is zero, with the final term \eqref{eq: 4th line?} disappearing.
\end{remark}
\section{Quantitative uniqueness}
\label{ch: quant uniqueness}
\subsection{Proof of Theorem \ref{thrm: c 1 alpha cost Linfty bound}: $L^\infty$ bound for $\mathcal{C}^{1, \alpha}$ cost}
Fix $\rho \in \PP(\X)$ and $\mu \in \PP(\Y)$. Since $c \in \mathcal{C}^{1, \alpha}(\X \times \Y)$ is uniformly bounded and continuous, it satisfies Assumption \ref{ass: existence hypotheses} so that we have existence. Let $\phi_0, \phi_1 \in \Phi_c(\rho, \mu)$ be two Kantorovich potentials. We fix some arbitrary reference point $x' \in P_x(\Gamma)$, then up to choosing $l \in \R$ in the diameter definition \eqref{eq: diameter definition}, we may assume that $\phi_0(x') = \phi_1(x')$. Let $x \in P_x(\Gamma)$ be arbitrary, then there exist $z', z \in \Omega$ such that
\begin{equation*}
    \|x - z\|, \|x' - z'\| \leq 2d_H,
\end{equation*}
where we denote $d_H = d_H(\spt \rho, \Omega)$. Here we use the density of $P_x(\Gamma) \subseteq \spt \rho$ so that it satisfies the same Hausdorff bound. The factor $2d_H$ rather than $d_H$ is for this reason, as $P_x(\Gamma)$ is not closed. We could take $d_H + \varepsilon$ for any $\varepsilon>0$. By Assumption \ref{ass: domain assumption connexivity by bounded curvature arcs}, let $\omega : [0, 1] \to \Omega$ be a curve from $\omega(0) = z'$ to $\omega(1) = z$ with $    \|\dot \omega\|_{L^\infty} + |\ddot \omega|([0, 1]) \leq C_\Omega$. Since $\omega(t) \in \Omega$ for all $t$, for any $n \geq 2$, there exist $\{x_i\}_{i=0}^n \subseteq P_x(\Gamma)$ with $x_0 = x', x_n = x$ and such that
\begin{equation}
\label{eq: points along the curve close to x i l infty c 1 apha proof}
    \forall i = 0,..., n; \quad \|\omega(i/n) - x_i\| \leq 2d_H.
\end{equation}
Take $\{y_i\}_{i=0}^n \subseteq \Y$ such that $\{(x_i, y_i)\}_{i=0}^n \subseteq \spt \Gamma$. Applying Lemma \ref{lem: michael curve approx} combined with \eqref{eq: points along the curve close to x i l infty c 1 apha proof}, we deduce that for any $n \geq 2$,
\begin{align}
    \nonumber
    \sum_{i=0}^{n-1} \big[c(x_{i+1},& y_i) - c(x_i, y_i) + c(x_i, y_{i+1}) - c(x_{i+1}, y_{i+1})\big]\\ &\leq C(\X, \Y, c) \left( n d_H + \frac{C_\Omega}{n} + \frac{C_\Omega^{\alpha+1}}{n^{\alpha}} \right) \\
    &\leq C(\X, \Y, c)(1+ C_\Omega)^2\left( n d_H + \frac{1}{n^{\alpha}} \right), \label{eq: bound n dh + 1/n alpha in c 1 alpha proof}
\end{align}
for some $C(\X, \Y, c)>0$ depending on $\|\nabla_x c\|_{\mathcal{C}^0(\X \times \Y)}$ and $[\nabla_x c]_{\X \times \Y, \alpha}$. If $d_H=0$, taking $n \to \infty$ we deduce the diameter is zero, i.e. we have uniqueness up to a constant. If $d_H = +\infty$, the bound we are trying to prove is trivial. Otherwise, optimising over $n \geq 2$, we choose
\begin{equation*}
    n = 1 + \left\lceil d_H^{-1/(1+\alpha)} \right\rceil
\end{equation*}
where $\lceil x \rceil$ denotes the smallest integer greater than or equal to $x \in \R$. Then
\begin{align*}
    n d_H + \frac{1}{n^\alpha} \leq& \left(d_H^{-1/(1+ \alpha)} + 2\right) d_H + \frac{1}{\left(d_H^{-1/(1+\alpha)}\right)^{\alpha}}\\ \leq&  2 d_H^{1- 1/(\alpha+1)} + 2 d_H.
\end{align*}
Consequently, returning to \eqref{eq: bound n dh + 1/n alpha in c 1 alpha proof}, we have
\begin{equation*}
    \sum_{i=0}^{n-1} \big[c(x_{i+1}, y_i) - c(x_i, y_i) + c(x_i, y_{i+1}) - c(x_{i+1}, y_{i+1})\big] \leq C(1+C_\Omega)^2 \left(d_H^{1-1/(\alpha + 1)} + d_H\right)
\end{equation*}
for some $C(\X, \Y, c) > 0$. Since $x \in P_x(\Gamma)$ was arbitrary, and $P_x(\Gamma)$ has full $\rho$ mass, we are done by Proposition \ref{prop: the direct bound expression}. \qedwhite

\subsection{Proof of Theorem \ref{thrm: quant L^p moment bounds}: $L^q$ bound for $\rho$ unbounded, quadratic cost}
By Section \ref{sect: equiv quadratic and bilinear}/ Remark \ref{rmk: invariance under cost translations}, it is sufficient to prove the result for the cost $c_l(x, y) = - \langle x, y \rangle$ and the associated Brenier potentials. Here
\begin{equation*}
    \nabla_x c(x, y) = -y \quad \text{ and } \quad \sup_{y \in \Y} [\nabla_x c]_{\X, 1} = 0.
\end{equation*}
Let $\rho \in \PP_2(\R^d)$ be such that $\int \|x\|^{q/2} \di\rho \leq M$, let $\mu \in \PP(\Y)$ be arbitrary. Then, since $\Y$ is bounded, Assumption \ref{ass: existence hypotheses} is satisfied so that $\Phi_{c}(\rho,\mu)$ is non-empty. Let $z' \in \Omega$ be a centre of the star-shaped domain $\Omega$. Then
\begin{align}
    \int_{\R^d} \| x - z'\|^{q/2} \di \rho(x) \leq& \max\left(2^{q/2-1}, 1\right) \left(\int_{\R^d} \|x\|^{q/2} \di\rho(x) + \|z' \|^{q/2} \right)\nonumber \\ 
    \leq& C_q(M + \|z'\|^{q/2}) := M'.\label{eq: M dash defn in starlike unbounded proof}
\end{align} 
Here $\Y$ is compact so that $P_x(\Gamma) = \spt \rho$ by Lemma \ref{lem: projection is full support when mu bounded}. Fix $x' \in \spt \rho$ such that
\begin{equation*}
    \|x' - z'\| \leq d_H,
\end{equation*}
where $d_H = d_H(\spt \rho, \Omega)<+ \infty$, otherwise the bound we want to prove is trivial. Let $\phi_0, \phi_1 \in \Phi_c(\rho, \mu)$ be two $\rho$-side Kantorovich potentials, up to choosing $l \in \R$ in the diameter definition \eqref{eq: diameter definition}, we may assume that $\phi_0(x') = \phi_1(x')$. For any $x \in \spt \rho$, there exists $z \in \Omega$ with
\begin{equation*}
    \|x - z\| \leq d_H.
\end{equation*}
We use the curve
\begin{equation*}
    \omega :[0, 1] \to \Omega; \quad \omega(t) = (1-t)z' + t z,
\end{equation*}
which is in $\Omega$ by the star-shaped assumption. We have $\| \dot \omega \|_{L^\infty} = \|z-z'\| \leq \|x -z' \| + d_H$ and $|\ddot \omega|([0, 1]) = 0$. For any $n \geq 2$, there exist $\{x_i\}_{i=0}^n \subset \spt \rho$ with $x_0 = x', x_n = x$ and such that
\begin{equation*}
    \forall i = 0,..., n; \quad \|\omega(i/n) - x_i\| \leq 2d_H.
\end{equation*}
Take $\{y_i\}_{i=0}^n \subseteq \Y$ such that $\{(x_i, y_i)\}_{i=0}^n \subseteq \spt \Gamma$. Then, by Lemma \ref{lem: michael curve approx} and Proposition \ref{prop: the direct bound expression} we have
\begin{equation}
\label{eq: raw bound in n moment bound proof}
    \forall x \in \spt \rho,\; \forall n \geq 2, \quad|\phi_0(x) - \phi_1(x)| \leq R_\Y\left(n d_H + \frac{\|x-z'\| + d_H}{n}\right),
\end{equation}
where $R_\Y>$ denotes the smallest radius such that $\Y \subset B_{R_\Y}(0)$. If $d_H=0$, taking $n \to \infty$ we deduce that the diameter is zero. Otherwise, optimising over $n \geq 2$, we choose $n = \max(2, \lceil n_* \rceil)$ where
\begin{equation*}
    n_* = \sqrt{\frac{\|x-z'\| + d_H}{d_H}}\geq 1.
\end{equation*}
Then
\begin{align*}
    n d_H + \frac{\|x-z'\| + d_H}{n} \leq& \left( \sqrt{\frac{\|x-z'\| + d_H}{d_H}}+ 1\right) d_H + (\|x-z'\| + d_H) \sqrt{\frac{d_H}{\|x-z'\| + d_H}}\\
    =& 2 \sqrt{d_H}\sqrt{\|x-z'\| + d_H} + d_H \leq 3\left(\sqrt{d_H}\sqrt{\|x-z'\|} + d_H\right).
\end{align*}
where we used in the last line that $\sqrt{a+ b} \leq \sqrt a + \sqrt b$ for $a, b \geq 0$. For $q \geq1$, we have the inequality
\begin{equation*}
\label{eq: q power inequality}
    a^{q} + b^{q} \leq (a + b)^{q} \leq 2^{q-1}(a^{q} + b^{q}) \quad \forall a, b \geq 0.
\end{equation*}
Consequently, returning to \eqref{eq: raw bound in n moment bound proof}, we have that for all $x \in \spt \rho$,
\begin{equation*}
|\phi_0(x) - \phi_1(x)|^{q} \leq R_\Y^{q} 3^q\left( \sqrt{d_H}\sqrt{\|x-z'\|} + d_H \right)^{q} \leq R_\Y^{q} 6^q \left(d_H^{q/2} \|x-z'\|^{q/2} + d_H^{q}\right).
\end{equation*}
Integrating with respect to $\rho$ and using \eqref{eq: M dash defn in starlike unbounded proof}, we deduce
\begin{equation*}
    \|\phi_0 - \phi_1\|_{L^{q}(\rho)}^{q} \leq 6^qR_\Y^{q} M' \left(d_H^{q/2} + d_H^{q}\right) \leq 6^qR_\Y^{q} M'\left(d_H^{1/2} + d_H\right)^{q}.
\end{equation*}
Taking the $q$-th root and using that $M'^{1/q} \leq 2(M^{1/q} + \|z'\|^{1/2}) \leq 2(1+M + \|z'\|^{1/2})$, we have
\begin{equation*}
    \|\phi_0 - \phi_1\|_{L^q(\rho)}\leq C(d_H^{1/2} + d_H)
\end{equation*}
For some constant $C(M,  z', \Y)>0$ as required. \qedwhite

\subsection{Proof of Theorem \ref{thrm: quantitative uniqueness for measures on uniform grids}: $L^\infty$ bound for measures on uniform grids, quadratic cost}
We first provide some calculations related to Lemma \ref{lem: michael curve approx}, which motivate an assumption on sets $\Omega \subseteq \R^d$ for which the conclusion of Theorem \ref{thrm: quantitative uniqueness for measures on uniform grids} holds. Let $\omega: [0, 1] \to \R^d$ be given by
\begin{equation*}
    \omega(t) = x + t v
\end{equation*}
for some $x, v \in \R^d$. We consider the cost $c_l(x, y) = - \langle x, y \rangle$, points $x_i = \omega(i/n)$ for $i = 0,...,n$, and $\{y_i\}_{i=0}^n \subseteq \Y \subseteq \R^d$. Here, $x_{i+1} - x_i = x_1 - x_0 = v/n$ for all $i = 0,...,n-1$, so that the quantity \eqref{eq: the direct bound expression} is
\begin{equation*}
\sum_{i=0}^{n-1} \langle x_{i+1} - x_i, y_{i+1} - y_i \rangle = \langle x_1 - x_0, \sum_{i=0}^{n-1} y_{i+1} - y_i \rangle = \langle x_1 - x_0, y_n - y_0 \rangle \leq \|x_1 - x_0\| \diam(\Y).
\end{equation*}
The fact that the sum is telescopic is precisely the fact that the second differences \eqref{eq: second difference of the cost term} are zero, this is a special case of Lemma \ref{lem: michael curve approx}. Now, for $\{x^k\}_{k=0}^N, \{v^k\}_{k=1}^N \subseteq \R^d$, take $N$ such affine curves affine curves
\begin{equation*}
    \omega_k : [0, 1] \to \R^d, \quad \omega_k(t) = x^k + t v^k,
\end{equation*}
such that $\omega_k(1) = \omega_{k+1}(0)$ for $k = 1,...,N-1$. Define the continuous curve $\omega : [0, 1] \to \R^d$ as their concatenation, so that
\begin{equation*}
    \forall k = 1,...,N, \; \forall t \in \left[ \frac{k-1}{N}, \frac{k}{N}\right], \quad \omega(t) = \omega_k(Nt - k +1).
\end{equation*}
We take $\{x_i\}_{i=1}^n$ to be successions of $n_k$ equidistant points along each $\omega_k$ in succession, so that for indices $i$ corresponding to points $x_i$ along the curve $\omega_k$, the difference $x_{i+1} - x_i = v_k/n_k$ is constant. Applying Lemma \ref{lem: michael curve approx} on each segment, we have
\begin{equation}
\label{eq: grid bound calcul 2}
    \sum_{i=0}^{n-1} \langle x_{i+1} - x_i, y_{i+1} - y_i \rangle \leq \sum_{k=1}^N \frac{\diam \Y \|v_k\|}{n_k} \leq N \diam \Y \max_{i=0}^n\|x_{i+1} - x_{i} \|
\end{equation}
To turn the above into a diameter bound, we ask that any two grid points in $\Omega \cap \varepsilon \mathbb{Z}^d$ can be joined by a path that follows a uniformly bounded number of straight grid directions, with step size $O(\varepsilon)$.
\begin{assumption}[Finite-turn grid connectivity]
\label{ass: domain assumption for grid discretisation}
Suppose $\Omega \subseteq \R^d$ is such that there exist $N(\Omega) \in \mathbb{N}$ and $S(\Omega)>0$ with the following property: For all $\varepsilon> 0$ and all $z, z' \in \Omega \cap \varepsilon \mathbb{Z}^d$, there exists a grid path $\{x_i\}_{i=0}^n \subseteq \Omega \cap \varepsilon \mathbb{Z}^d$ with $x_0 = z'$ and $x_n = z$, such that:
\begin{enumerate}[label=(\roman*)]
    \item $x_{i+1} - x_{i}$ is piecewise constant throughout $i = 0,..., n-1$, changing at most $N$ times.
    \item The step size is uniformly bounded by $\|x_{i+1} - x_i\| \leq S \varepsilon$.
\end{enumerate}
\end{assumption}
\begin{proposition}
    Let $\Omega = I_1 \times \cdots I_d \subseteq \R^d$ for (potentially unbounded) intervals $I_1,...,I_d \subseteq \R$. Then $\Omega$ satisfies Assumption \ref{ass: domain assumption for grid discretisation} with $N(\Omega) = d$ and $S(\Omega) = 1$.
\end{proposition}
\begin{proof}
We simply follow $d$ successive lines in Cartesian directions in the grid. We connect any two points $z, z' \in \Omega \cap \varepsilon \mathbb{Z}^d$ using $\{x_i\}_{i=0}^n$, $x_0 = z'$, $x_n = z$ such that denoting $\{e_j\}_{j=1}^d$ the standard orthonormal frame of $\R^d$, $x_{i+1} - x_i \in \varepsilon \{\pm e_j\}_{j=1}^d$, so that $S=1$, and we only need to follow the concatenation of at most $N= d$ different affine grid paths.
\end{proof}
\begin{proof}[Proof of Theorem \ref{thrm: quantitative uniqueness for measures on uniform grids}]
We prove that the conclusion of Theorem \ref{thrm: quantitative uniqueness for measures on uniform grids} holds (with a different constant) for any $\Omega \subseteq \R^d$ satisfying Assumption \ref{ass: domain assumption for grid discretisation}.  Let $\phi_0, \phi_1 \in \Phi_c(\rho, \mu)$, then up to choosing $l \in \R$ in the diameter definition \eqref{eq: diameter definition}, we may assume $\phi_0(x') = \phi_1(x')$ for some $x' \in \spt \rho$. Let $x \in \spt \rho$ be arbitrary, take $\{x_i\}_{i=0}^n \subseteq \spt \rho$ a path from $x'$ to $x$ satisfying the hypotheses of Assumption \ref{ass: domain assumption for grid discretisation}, and take $\{y_i\}_{i=0}^n \subseteq \Y$ with $\{(x_i, y_i)\}_{i=0}^n \subseteq \spt \Gamma$. Then
\begin{equation*}
    \max_{i=0}^{n-1} \| x_{i+1} - x_{i}\| \leq S \varepsilon
\end{equation*}
and hence the bound \eqref{eq: grid bound calcul 2} becomes
\begin{equation*}
    \sum_{i=0}^{n-1} \langle x_{i+1} - x_i, y_{i+1} - y_i \rangle \leq N S \varepsilon \diam(\Y).
\end{equation*}
Hence, by Proposition \ref{prop: the direct bound expression}, Theorem \ref{thrm: quantitative uniqueness for measures on uniform grids} holds with constant $C(\Omega, \Y) = N(\Omega) S(\Omega) \diam \Y >0$. (For the Quadratic Kantorovich potentials, the constant is doubled, due to the relation \eqref{eq: relation between quadratic and max corr potential sets}.)
\end{proof}

\subsection{Example: Optimality of the exponent $1$ for the grid case}
\label{sect: optimality of exponent in grid case}
The following example demonstrates that, in general, we cannot have quantitative uniqueness with a better exponent than $1$ in Theorem \ref{thrm: c 1 alpha cost Linfty bound}, and consequently Theorem \ref{thrm: quantitative uniqueness for measures on uniform grids} is sharp in its exponent. Let $\Omega = \X = \Y= [0, 1]^d$, and for $n \in \mathbb{N}$, let $\mathcal{G}_n = [0, 1]^d \cap \left( \frac{1}{n} \mathbb{Z}^d \right)$. Define the discrete probability measures
\begin{equation*}
    \rho_n := \frac{1}{(n+1)^d}\sum_{x \in \mathcal{G}_n}  \delta_x,\quad \quad \mu_n := \frac{n}{n+1} \delta_{(0,...,0)} + \frac{1}{n+1} \delta_{(0,...,0,1)},
\end{equation*}
then for each $n \in \mathbb{N}$, $\spt \rho_n = \mathcal{G}_n$. Denote the $x_d$-face of the cube $[0, 1]^d$ by $F = [0, 1]^{d-1} \times \{1\}$. Then the primal quadratic optimal transport between $\rho_n$ and $\mu_n$ has a unique solution, which transports mass via
\begin{equation}
\label{eq: optimal plan in the grid example}
    F \cap \spt \rho_n \mapsto (0,...,0,1) \quad \text{ and } \quad \spt \rho_n \setminus F \mapsto (0,...,0).
\end{equation}
Hence by the compatibility condition \eqref{intro: compatibility condition}, both
\begin{align*}
    \phi^0_n: [0,1]^d & \to \R \quad \quad & \quad \quad \phi^1_n: [0,1]^d & \to \R\\
    (x_1,x_2,...,x_d) &\mapsto \left(x_d - \frac{n-1}{n}\right)_+ & (x_1,x_2,...,x_d) &\mapsto \left(x_d - 1\right)_+
\end{align*}
are optimal Brenier potentials. Since $\phi^0_n(0) = \phi_n^1(0)$ and
\begin{equation*}
    \big|\phi_n^0((0,...,0, 1)) - \phi_n^1((0,...,0, 1))\big| = \frac{1}{n},
\end{equation*}
we have
\begin{equation*}
    \diam_{L^\infty}(\Phi_{c_l}(\rho_n, \mu_n)) \geq \inf_{l \in \R} \|\phi_n^0 - \phi_n^1 - l\|_{L^\infty(\rho_n)} \geq \frac{1}{2n}.
\end{equation*}
It follows that the exponent in Theorem \ref{thrm: quantitative uniqueness for measures on uniform grids} is sharp in $\varepsilon$. Since $d_H(\spt \rho_n, [0, 1]^d)= \frac{\sqrt{d}}{2n}$, and $\Omega = [0, 1]^d$ satisfies Assumption \ref{ass: domain assumption connexivity by bounded curvature arcs}, the optimal exponent for bounds of the form given in Theorem \ref{thrm: c 1 alpha cost Linfty bound} cannot be better than $1$.

\subsection{Proof of Theorem \ref{thrm: quant stochastic}: Stochastic $L^\infty$ bound, empirical source measure}
We use the following theorem from \cite{reznikov2016covering}.
\begin{theorem}\textnormal{\cite[Theorem 2.1]{reznikov2016covering}}
\label{thrm: stochastic covering radii}
Let $\rho$ be a probability measure on $\R^d$ for which there exist constants $r_0, s, C >0$ such that for all $x \in \spt \rho$ and $r \leq r_0$,
\begin{equation*}
    \rho(B_r(x)) \geq Cr^s.
\end{equation*}
Then there exist constants $C_0, C_1, C_2, C_3, \beta_0> 0$ depending on $\rho$ such that 
\begin{equation}
    \forall N \in \mathbb{N}, \quad \mathbb{E} \left[d_H\left(\{X_i\}_{i=1}^N, \spt \rho\right) \right] \leq C_0 \left(\frac{\log N}{N}\right)^{1/s},
\end{equation}
and for all $\beta > \beta_0$,
\begin{equation}
\label{eq: concentration bound radii}
    \forall N \in \mathbb{N}, \quad\mathbb{P}\left[ d_H\left(\{X_i\}_{i=1}^N, \spt \rho\right) \geq C_1 \left(\frac{\beta \log{N}}{N}\right)^{1/s}\right] \leq C_2 N^{1-C_3 \beta}.
\end{equation}
\end{theorem}
\begin{proof}[Proof of Theorem \ref{thrm: quant stochastic}]
    Fix $\mu \in \PP(\Y)$. By Theorem \ref{thrm: c 1 alpha cost Linfty bound} with $\Omega = \spt \rho$, there exists a constant $C(\spt \rho, \X, \Y, c)>0$ such that
    \begin{equation}
    \label{eq: diameter bound pre expectation}
        \forall N \in \N, \quad \diam_{L^\infty} (\Phi_c(\rho_N, \mu)) \leq C d_H\left(\spt \rho, \{X_i\}_{i=1}^N\right)^{\alpha/(1 + \alpha)}.
    \end{equation}
    Here we used that
    \begin{equation*}
        d_H\left(\spt \rho, \{X_i\}_{i=1}^N\right) \leq (\diam\spt \rho)^{1/(1 + \alpha)}d_H\left(\spt \rho, \{X_i\}_{i=1}^N\right)^{\alpha/(1+\alpha)},
    \end{equation*}
     and that $\diam\spt \rho < + \infty$ by Assumption \ref{ass: domain assumption connexivity by bounded curvature arcs}. Taking the expectation of \eqref{eq: diameter bound pre expectation}, by Jensen's inequality, we have
    \begin{equation*}
        \mathbb{E} \left[d_H\left(\{X_i\}_{i=1}^N, \spt \rho\right)^{\alpha/(1+\alpha)} \right] \leq \mathbb{E} \left[d_H\left(\{X_i\}_{i=1}^N, \spt \rho\right) \right]^{\alpha/(1+\alpha)},
    \end{equation*}
    from which the bound \eqref{eq: expectation bound diameter general case} follows. For any $\varepsilon>0$,
    \begin{align*}
        \mathbb{P}\left[ \diam_{L^\infty} (\Phi_c(\rho_N, \mu)) \geq \varepsilon\right] \leq& \mathbb{P}\left[C d_H\left(\spt \rho, \{X_i\}_{i=1}^N\right)^{\alpha/(1+\alpha)} \geq \varepsilon\right]\\ =& \mathbb{P}\left[d_H\left(\spt \rho, \{X_i\}_{i=1}^N\right) \geq \left(\frac{\varepsilon}{C}\right)^{(1+ \alpha)/\alpha}\right].
    \end{align*}
    Consequently, by Theorem \ref{thrm: stochastic covering radii}, we deduce \eqref{eq: concentration bound for the diameter} with $C_1 = C C_1'^{\alpha/(1+\alpha)}$ where $C_1'$ is the constant from \eqref{eq: concentration bound radii}, and $C_2, C_3, \beta_0>0$ have the same values.
\end{proof}

% \listoftodos

\appendix
\section{Hölder regularity of the $p$-cost for $p \in (1, 2]$}
\label{app: regularity of p cost}
\begin{lemma}
    Let $h: \R^d \to \R$, $h(z) = \|z\|^p$ for some $p \in (1, 2]$. Then $h \in \mathcal{C}^{1, p-1}_{\loc}\big(\R^d\big)$.
\end{lemma}
\begin{proof}
We use the proof given in \cite[Section 3.3]{mischler2024quantitative}, see also \cite{lindqvist2019notes} for related inequalities. For $z \neq 0$, we have
\begin{equation*}
    \nabla h(z) = p \|z\|^{p-2} z.
\end{equation*}
so that $\nabla h$ is locally bounded and continuous in $z$. Let $z, z' \in \R^d$, then
\begin{equation}
\label{eq: appendix p cost bound}
    \big\| \|z\|^{p-2} z - \|z'\|^{p-2} z' \big \|^2 = \big \| \| z \|^{p-1} - \|z' \|^{p-1} \big\|^2 + 2 \|z \|^{p-2} \|z'\|^{p-2} \big( \|z\| \| z'\| - \langle z, z' \rangle \big).
\end{equation}
The first term of the right hand side of \eqref{eq: appendix p cost bound} is controlled using that the function $t \to t^{p-1}$ is $(p-1)$ Hölder continuous on $\R_+$, giving
\begin{equation*}
    \big \| \| z \|^{p-1} - \|z' \|^{p-1} \|^2 \leq \Big| \|z\| - \|z'\| \Big|^{2(p-1)} \leq \|z-z'\|^{2(p-1)},
\end{equation*}
where the second inequality follows from the triangle inequality. For the second term in \eqref{eq: appendix p cost bound}, we have (assuming $z \neq z'$)
\begin{align*}
    \|z \|^{p-2} \|z'\|^{p-2} \big( \|z\| \| z'\| - \langle z, z' \rangle \big) =& \|z \|^{p-1} \|z'\|^{p-1} \left( 1- \frac{\langle z, z' \rangle}{\|z\| \| z'\|} \right)^{p-1}\left( 1- \frac{\langle z, z' \rangle}{\|z\| \| z'\|} \right)^{2-p}\\
    \leq&2^{2-p}\left( \|z\| \| z'\|- \langle z, z' \rangle \right)^{p-1}\\
    \leq&  2^{3-2p} \|z-z'\|^{2(p-1)}
\end{align*}
where we used in the last line that $\|z\| \| z'\|- \langle z, z' \rangle \leq \frac{1}{2}\|z\|^2 + \frac{1}{2}\|z'\|^2 - \langle z, z' \rangle = \frac{1}{2}\| z- z'\|^2$. Thus, the left hand side of \eqref{eq: appendix p cost bound} is bounded by $C\|z-z'\|^{2(p-1)}$ for some $C(p)>0$ so that $\nabla h$ is $p-1$ Hölder continuous as required.
\end{proof}
\section{Selecting $c$-concave representatives}
\label{ch: precisions on c concavity}

The primal problem \eqref{intro: c OT problem} is only concerned with the values of the cost $c$ on $\spt \rho \times \spt \mu$. Consequently, for any $A \supseteq \spt \rho$ and $B \supseteq \spt \mu$, one could define the primal problem over $A \times B$ rather than $\R^d \times \R^d$, and arrive at a dual problem with constraint \eqref{intro: dual constraint} over $A \times B$ instead. Instead of taking the $c$-transforms \eqref{intro: c transf definition} with respect to $\spt \rho$ and $\spt \mu$, we could similarly have defined $(c, A)$ and $(\overline c, B)$ transforms - there is no canonical choice, and different sets $A, B$ give fundamentally different transforms. We emphasise that these transforms act on pointwise-defined functions, and changing the value of a function on $\rho$ null sets can drastically change the resulting transform.

The proposition below says that the dual problem on $\R^d \times \R^d$ and on $\spt \rho \times \spt \mu$ are ``essentially the same problem", and that both these problems are ``essentially the same problem" as searching for $c$-concave optimisers. The same result holds more generally for any $A \supseteq \spt \rho$ and $B \supseteq \spt \mu$ and $(c, A, B)$-concave potentials.
\begin{proposition}
\label{prop: selection principle for c concave representative}
Suppose that $c$ satisfies Assumption \ref{ass: existence hypotheses}.
\begin{enumerate}[label=(\roman*)]
    \item Given $(\phi, \psi)$ optimal for \eqref{intro: c dual problem} on $\R^d \times \R^d$, there exists a $(c, \spt \rho, \spt \mu)$-concave pair $(\tilde \phi, \tilde \psi)$ with $\phi = \tilde \phi$ $\rho$-a.e. and $\psi = \tilde \psi$ $\mu$-a.e.
    \item Given $(\phi,\psi)$ optimal for the dual problem \eqref{intro: c dual problem} on $\spt \rho \times \spt \mu$, there exists $(\tilde \phi, \tilde \psi)$ satisfying the dual constraint \eqref{intro: dual constraint} on all $\R^d \times \R^d$, such that $\phi = \tilde \phi$ $\rho$-a.e. and $\psi = \tilde \psi$ $\mu$-a.e.
\end{enumerate}
\end{proposition}
\begin{proof}
    Take $\phi, \psi : \R^d \to \overline \R$ in $L^1$, optimal for the dual problem \eqref{intro: c dual problem} over $\R^d \times \R^d$. Then, they are also optimal for the dual problem over $\R^d \times \spt \mu$ (they obtain the same value in the objective, and the values of both these problems are the same). By the definition of the transform, the dual constraint \eqref{intro: dual constraint} gives
    \begin{equation}
    \label{eq: pointwise comparison between c transformed functions}
        \phi(x) \leq \psi^{\overline c, \spt \mu}(x) \quad \forall x \in \R^d.
    \end{equation}
    We will first show that $\psi^{\overline c, \spt \mu} \in L^1(\rho)$. By Assumption \ref{ass: existence hypotheses}, there exist real valued functions $a \in L^1(\rho)$ and $b \in L^1(\mu)$ such that $|c(x, y)| \leq a(x) + b(y)$ for all $(x, y) \in \R^d \times \R^d$ and hence
    \begin{equation*}
        \forall x \in \R^d\; \forall y \in \spt \mu , \quad \psi^{\overline c, \spt \mu}(x) \leq c(x, y) - \psi(y) \leq a(x) + b(y) - \psi(y).
    \end{equation*}
    Since $b, \psi \in L^1(\mu)$ there exists at least one point $y_0 \in \spt \mu$ at which both are finite, and hence
    \begin{equation*}
    \psi^{\overline c, \spt \mu}(x) \leq a(x) + b(y_0) - \psi(y_0) = a(x) + C.
    \end{equation*}
    It follows that $|\psi^{\overline c, \spt \mu}(x)| \leq \max\left(|\phi(x)|, |a(x) + C| \right)$ and hence $\psi^{\overline c, \spt \mu} \in L^1(\rho)$. Thus $(\psi^{\overline c, \spt \mu}, \psi)$ are admissible (and optimal) for the dual problem on $\R^d \times \spt \mu$, and so given some $\gamma \in \Gamma_c(\rho, \mu)$
    \begin{equation}
        \int_{\R^d \times \R^d} \psi^{\overline c, \spt \mu}(x) + \psi(y) \di \gamma(x, y) = \int_{\R^d \times \R^d} \phi(x) + \psi(y) \di \gamma(x, y).
    \end{equation}
    Since $\psi^{\overline c, \spt \mu}, \phi \in L^1(\rho)$ and $\psi \in L^1(\mu)$, we can use linearity of the integral to deduce
    \begin{equation}
    \label{eq: equality of c trans optims}
        \int_{\R^d} \phi(x) \di \rho(x) = \int_{\R^d} \psi^{\overline c, \spt \mu}(x) \di \rho(x).
    \end{equation}
    Combining \eqref{eq: pointwise comparison between c transformed functions} and \eqref{eq: equality of c trans optims}, we deduce $\phi = \psi^{\overline{c}, \spt \mu}$ $\rho$-a.e. By a symmetric argument, one deduces $(\psi^{\overline c, \spt \mu})^{c, \spt \rho} = \psi$ $\mu$-a.e. which gives a pair of $(c, \spt \rho, \spt \mu)$ representatives, proving the first claim.

    The proof of the second statement, is very similar. Given a pair of Kantorovich potentials $(\phi, \psi)$ defined only on $\spt \rho \times \spt \mu$, $(\psi^{\overline c, \spt \mu}, \psi)$ are Kantorovich potentials on $\R^d \times \spt \mu$ with $\psi^{\overline c, \spt \mu} = \phi$ $\rho$-a.e by the same logic. Then $(\psi^{\overline c, \spt \mu}, (\psi^{\overline c, \spt \mu})^{c, \R^d})$ are a pair of Kantorovich potentials which satisfy the dual constraint \eqref{intro: dual constraint} on all $\R^d \times \R^d$, whilst being $\rho$ and $\mu$ a.e. equal to the original pair.
\end{proof}

\section{Sets satisfying Assumption \ref{ass: domain assumption connexivity by bounded curvature arcs}}
\label{ch: path assumption}

\begin{proposition}
\label{prop: path regularity for various domains}
    Let $\Omega \subseteq \R^d$ be a connected, bounded set which can be written as a finite union of bounded, $\mathcal{C}^0$ domains (domains whose boundary is locally the graph of a continuous function) and $\mathcal{C}^{1,1}$ compact connected submanifolds of $\R^d$ with boundary. Then $\Omega$ satisfies Assumption \ref{ass: domain assumption connexivity by bounded curvature arcs}.
\end{proposition}

\begin{proof}[Proof of Assumption \ref{ass: domain assumption connexivity by bounded curvature arcs} for bounded $\mathcal{C}^0$ domains]
    A $\mathcal{C}^0$ domain is a domain whose boundary can be locally represented as the graph of a continuous function, see \cite[Definition 1.2.1.1]{GrisvardPierre1980Bvpi}. Let $\Omega$ be a $\mathcal{C}^0$ domain, we cover the compact set $\overline \Omega$ with the open sets
    \begin{equation}
    \label{eq: open cover}
        \bigcup_{x \in \Int \Omega} B_{r_x}(x) \; \cup \bigcup_{i=0}^n Q_i,
    \end{equation}
    where $0 < r_x \leq \diam \Omega$ is such that $B_{r_x}(x) \subseteq \Omega$ and $\{Q_i\}_{i=0}^n$ are a finite family of open cubes $Q$ with the following properties: There exist local systems of orthonormal coordinates $\{z_1,...,z_d\}$ for $\R^d$,
    \begin{enumerate}
        \item In these coordinates, $Q = (-L, L)^d$.
        \item There exists a continuous function $g: \hat Q \to \R$ where
            \begin{equation*}
                \hat Q = \left\{(z_1,...,z_{d-1})\, :\, -L< z_i < L, \quad 1 \leq i \leq d-1 \right\}
            \end{equation*}
            and such that
            \begin{equation*}
                |g(\hat z)| \leq L/2 \; \text{ for each } \; \hat z = (z_1,...z_{d-1}) \in \hat Q,
            \end{equation*}
            \begin{equation*}
                \Omega \cap Q = \{ z = (\hat z, z_d) : z_d < g(\hat z) \},
            \end{equation*}
            \begin{equation*}
                \partial\Omega \cap Q = \{ z = (\hat z, z_d) : z_d = g(\hat z) \}.
            \end{equation*}
    \end{enumerate}
    We take a finite subcover of \eqref{eq: open cover}, then consider the graph $G$ whose vertices are the charts of our subcover, and which contains an edge when two charts intersect. This graph is connected and finite. We first show that any two points on one of these charts can be joined by an arc with the quantity \eqref{eq: connexivity of support by arcs of uniformly bounded curvature} uniformly bounded. For any $ x \in \Omega$, two points $z, z' \in B_{r_x}(x)$ can be joined by $\omega(t) = (1-t) z + t z'$, with $\|\dot \omega\|_{L^\infty} = \|z'-z\| \leq \diam \Omega$ and $|\ddot \omega|([0, 1]) = 0$. For a boundary chart $Q$ and any $z, z' \in \Omega \cap Q$, we construct a curve joining them whose length and curvature are bounded independently of these points. Setting $ z = (\hat z, z_d)$ and $z' = (\hat z',  z_d')$ in coordinates, we follow a piecewise linear interpolation
    \begin{equation}
    \label{eq: path connecting c0 boundary}
        \omega(t) = \begin{cases}
            (\hat z, z_d + 3t(-L/2 - z_d)) & t \in [0, 1/3]\\
            (\hat z + (3t-1)(\hat z' - \hat z), -L/2) & t \in [1/3, 2/3]\\
            (\hat z', -L/2 + (3t-2)(z_d' + L/2) & t \in [2/3, 1],
        \end{cases}
    \end{equation}
    see Figure \ref{fig: path interpolation c0 boundary}. Then $\|\dot \omega\|_{L^\infty} \leq 3 \diam \Omega$, and
    \begin{equation}
    \label{eq: bounding tv app}
        |\ddot \omega|([0, 1]) = \|\dot \omega(1/3^+) - \dot \omega(1/3^-)\| + \|\dot \omega(2/3^+) - \dot \omega(2/3^-)\| \leq 4 \| \dot \omega\|_{L^\infty} \leq 12 \diam \Omega,
    \end{equation}
    so that any boundary chart $Q$ can be connected by arcs with the quantity \eqref{eq: connexivity of support by arcs of uniformly bounded curvature} uniformly bounded.
\begin{figure}[ht]
    \centering
\begin{tikzpicture}[scale=0.7, >=Latex]

\def\W{7.2} 
\def\H{4.8}  
\def\xL{0}
\def\xR{\W}
\def\yB{0}
\def\yT{\H}

\def\yMid{2.4}
\def\yLow{1.4}
\def\yHigh{3.4}

\def\xa{0.7}
\def\xb{6.4}
\def\ya{2.95}
\def\yb{2.82}

\def\yint{1.4}

\tikzset{
box/.style={dashed, line width=0.9pt},
bdry/.style={line width=1.1pt},
pathc/.style={line width=1.1pt},
strip/.style={dashed, line width=0.6pt, opacity=0.45},
shadeomega/.style={fill=black, fill opacity=0.06},
dot/.style={circle, fill, inner sep=1.4pt},
}

\draw[box] (\xL,\yB) rectangle (\xR,\yT);
\node[above right] at (\xR,\yT) {$Q$};

\draw[strip] (\xL,\yLow) -- (\xR,\yLow);
\draw[strip] (\xL,\yHigh) -- (\xR,\yHigh);

\node[left, opacity=0.65] at (\xL,\yHigh) {$L/2$};
\node[left, opacity=0.65] at (\xL,\yLow) {$-L/2$};

\fill[shadeomega]
(\xL,\yB) --
(\xR,\yB) --
(\xR,2.92)
.. controls (6.85,2.82) and (6.55,2.76) .. (\xb,\yb)
.. controls (5.65,2.92) and (5.15,3.05) .. (4.55,3.02)
.. controls (4.00,2.98) and (3.55,2.70) .. (2.95,2.62)
.. controls (2.25,2.56) and (1.85,2.84) .. (1.35,2.88)
.. controls (1.05,2.89) and (0.88,2.90) .. (\xa,\ya)
.. controls (0.45,3.00) and (0.18,3.06) .. (\xL,3.12)
-- (\xL,\yB) -- cycle;

\draw[bdry]
(\xL,3.12)
.. controls (0.45,3.00) and (0.88,2.90) .. (\xa,\ya)
.. controls (1.05,2.89) and (1.35,2.88) .. (1.75,2.74)
.. controls (2.05,2.64) and (2.45,2.60) .. (2.95,2.62)
.. controls (3.55,2.70) and (4.00,2.98) .. (4.55,3.02)
.. controls (5.15,3.05) and (5.65,2.92) .. (\xb,\yb)
.. controls (6.55,2.76) and (6.85,2.82) .. (\xR,2.92);

\draw[pathc] (\xa,\ya) -- (\xa,\yint) -- (\xb,\yint) -- (\xb,\yb);

\node[dot] at (\xa,\ya) {};
\node[dot] at (\xb,\yb) {};

\node[above left] at (\xa,\ya) {$z$};
\node[above] at (\xb,\yb) {$z'$};

\node at (3.6,0.95) {$\Omega\cap Q$};

\end{tikzpicture}
    \caption{The path \eqref{eq: path connecting c0 boundary} connecting points on the chart of a $\mathcal{C}^0$ boundary.}
    \label{fig: path interpolation c0 boundary}
\end{figure}
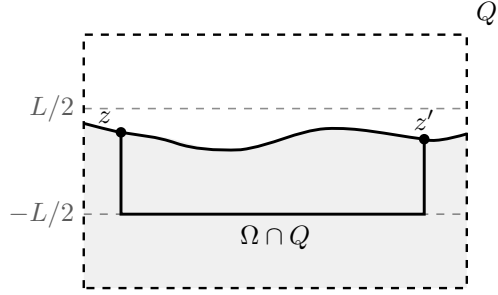
Constructing a path between arbitrary $z, z' \in \Omega$, we choose $\omega$ connecting $z$ to $z'$ to follow a concatenation of arcs which correspond to a simple path in the connectivity graph of charts, where each arc section has derivatives uniformly bounded. Then
    \begin{equation*}
        \|\dot \omega \|_{L^\infty} \leq C \diam \Omega
    \end{equation*}
    for some constant $C>0$ depending on $\diam G$ the graph diameter, and hence
    \begin{equation*}
        |\ddot \omega|([0, 1]) \leq C \diam \Omega
    \end{equation*}
    again for $C>0$ depending on $\diam G$. We are done since $z, z' \in \Omega$ were arbitrary.
\end{proof}
\begin{proof}[Proof of Assumption \ref{ass: domain assumption connexivity by bounded curvature arcs} for $\mathcal{C}^{1,1}$ compact connected submanifolds]
Let $s = \dim \Omega$. Similar to the $\mathcal{C}^0$ case, $\Omega$ can be covered by a finite number of $\mathcal{C}^{1, 1}$ charts, whose intersection graph is connected, and so it suffices to show Assumption \ref{ass: domain assumption connexivity by bounded curvature arcs} for a single chart. We consider an interior chart; boundary charts are treated identically. 

Let $f: B^s  \to  \Omega$ be a $\mathcal{C}^{1, 1}$ diffeomorphism between its image and $B^s = B^s_1(0) \subseteq \R^s$ the unit ball in dimension $s$. For $z, z' \in f(B^s) \subseteq \Omega$, let $w, w' \in B^s$ be such that $f(w) = z$ and $f(w') = z'$. Then we join $z$ and $z'$ by $\omega(t) = f((1-t)w + t w')$, so that
\begin{equation*}
    \|\dot \omega(t)\| =  \|(w' - w) \cdot \nabla f( (1-t)w + t w') \|   \leq 2\| \nabla f \|_{\mathcal{C}^0(B^s)}.
\end{equation*}
For $s, t \in [0,1]$,
\begin{equation*}
    \| \dot\omega(t) - \dot \omega(s)\| \leq |t-s| \|w - w'\|^2 [\nabla f]_{B_s, 1}
\end{equation*}
so that
\begin{equation*}
    |\ddot \omega|\big([0, 1]\big)  \leq 4 [\nabla f]_{B_s, 1},
\end{equation*}
which connects any two points by arcs with uniformly bounded derivatives as required.
\end{proof}
This clearly extends to a connected $\Omega \subseteq \R^d$ that can be written as a finite union of the above classes, by the same graph-diameter arguments and concatenation of paths.

\printbibliography

% \listoftodos

\end{document}